\documentclass{amsart}

\usepackage{The_Resolving_Completion_of_an_Exact_Category}

\title{The resolving completion of an exact category}
\date{\today}

\keywords{resolving subcategory, left heart, t-structure, effaceable functor, Ext-kernel, derived category, derived equivalence, exact category}

\subjclass[2020]{18G10, 18G80, 16E35, 18E10}

\author[M.~Lawson]{Marianne Lawson}
\address{Marianne Lawson,
Department of Mathematics,
University of Hamburg,
Bundesstr. 55,
20146 Hamburg,
Germany}
\email{marianne.lawson@uni-hamburg.de}

\author[J.~C.~Letz]{Janina C. Letz}
\address{Janina~C.~Letz,
Institute of Mathematics,
Paderborn University,
Warburger Str. 100,
33098 Paderborn,
Germany
}
\email{jletz@math.upb.de}

\author[J.~Sauter]{Julia Sauter}
\address{Julia Sauter,
Faculty of Mathematics,
Bielefeld University,
PO Box 100 131,
33501 Bielefeld,
Germany}
\email{jsauter@math.uni-bielefeld.de}

\begin{document}

\begin{abstract}
For an exact category we provide two constructions of an ambient category in which the initial category is resolving: In the derived category and in the Gabriel--Quillen embedding. For the first construction we describe a pre-aisle and its right orthogonal using different acyclicty conditions. We provide necessary and sufficient conditions when this pair is a t-structure. 
\end{abstract}

\maketitle

\setcounter{tocdepth}{1}
\tableofcontents

\section*{Introduction}

Exact categories appear in various areas of mathematics; for
example, in representation theory, algebraic geometry, algebraic topology or functional analysis. It is a common objective to find derived equivalences between such categories; meaning their derived categories are triangle equivalent. For example, this is at the heart of tilting theory \cite[Chapter~III]{Happel:1988},\cite{Rickard:1989},\cite[Chapter~5]{AngeleriHuegel/Happel/Krause:2007}. Many of these triangle equivalences are based on the derived equivalence obtained for a resolving subcategory by \cite{Henrard/vanRoosmalen:2020c}. 

For an exact category $\cat{E}$ we describe the \emph{resolving completion} $\cat{R}$ of $\cat{E}$; that is $\cat{E}$ can be identified with a resolving subcategory of $\cat{R}$ and it is universal in the following sense: Whenever $\cat{E} \subseteq \cat{F}$ is resolving, then $\mathcal{F}$ can be identified with a resolving subcategory of $\mathcal{R}$. We give two constructions for ambient categories of $\cat{E}$:
\begin{itemize}
\item We construct $\lheart{\cat{E}}$ as a full subcategory of $\dcat{\cat{E}}$, generalizing constructions of \cite{Huber:1995,Schneiders:1999,Henrard/Kvamme/vanRoosmalen/Wegner:2023}. It consists of complexes which are $\cat{E}$-acyclic in positive degrees and satisfy a weaker acyclicty condition, called left $\ext{\cat{E}}$-acyclicity, in negative degrees. This notion is inspired by \cite{Rump:2021}. 

\item We construct $\Rclosure{\cat{E}}$ as a fully exact subcategory of the localization of the functor category $\Mod{\cat{E}}$ at the category of locally effaceable functors $\Eff{\cat{E}}$. That is, we extend $\cat{E}$ in its Gabriel--Quillen embedding. This is inspired by \cite{Rump:2020,Henrard/Kvamme/vanRoosmalen/Wegner:2023}.
\end{itemize}

For these construction we show in \cref{LH_completion,Rcat_completion,Rcat_equals_LH} the following result:

\begin{introthm}
Let $\cat{E}$ be a weakly idempotent complete small exact category. Then both $\lheart{\cat{E}}$ and $\Rclosure{\cat{E}}$ are resolving completions of $\cat{E}$ and there is an equivalence $\lheart{\cat{E}} \to \Rclosure{\cat{E}}$. 
\end{introthm}

Similar to the constructions of $\lheart{\cat{E}}$ and $\Rclosure{\cat{E}}$ we also construct finite resolving completions $\lheart[b]{\cat{E}}$ and $\Rclosure[b]{\cat{E}}$, which yield equivalences of the bounded derived categories.

In the construction $\lheart{\cat{E}}$ we obtain a pair $(\cat{U},\cat{V})$, which is close to being a t-structure. In fact, when $\cat{E}$ has kernels, this is the t-structure described in \cite{Huber:1995,Schneiders:1999}. However, we do not expect the pair to be a t-structure in general, as the bounded derived category of an exact category need not support any t-structure by \cite{Neeman:2022}. Moreover, using this construction and its dual construction, we show in \cref{maxl_tpair} the following result:

\begin{introthm}
Let $\cat{E}$ be a weakly idempotent complete exact category. There exists a pair $(\cat{U},\cat{V})$ of full subcategories of $\dbcat{\cat{E}}$ such that 
\begin{enumerate}
\item $\cat{U}$ and $\cat{V}$ are closed under direct summands, isomorphisms and extensions; 
\item $\susp \cat{U} \subseteq \cat{U}$ and $\cat{V} \subseteq \susp \cat{V}$; and
\item $(\susp \cat{U})^{\perp} = \cat{V}$ and $\susp \cat{U} = {}^{\perp} \cat{V}$.
\end{enumerate}
\end{introthm}

In fact, we construct two pairs with these properties which in general are not equal. Each pair $(\cat{U},\cat{V})$ yields a \emph{maximally non-negative} extension $\cat{H} = \cat{U} \cap \cat{V}$ of $\cat{E}$; this means $\Hom_{\dbcat{\cat{E}}}(\cat{H},\susp^{<0} \cat{H}) = 0$ and $\cat{H}$ is maximal with the property.

The notion of left Ext-acyclicty is crucial in the constructions of $\lheart{\cat{E}}$ and $\Rclosure{\cat{E}}$. This is not surprising as it is preserved in a resolving extension. Left Ext-acyclicty encompasses two notions, namely left $\cat{E}$-acyclicity and Hom-acyclicity. The former is a slight generalization of $\cat{E}$-acyclicity, while the latter encodes acyclicity after application of the Yoneda functor. 

In order to obtain equivalences between the unbounded derived categories, we construct the \emph{$n$-resolving completion} of $\cat{E}$ for $n \geq 1$. 
Moreover, this construction preserves finiteness of the global dimension.
The $1$-resolving completion is the full subcategory of $\dcat{\cat{E}}$ consisting of two-term complexes represented by a monomorphism; this category has been studied by \cite{Schneiders:1999, Henrard/Kvamme/vanRoosmalen/Wegner:2023}. 

\begin{ack}
The authors were partly funded by the Deutsche Forschungsgemeinschaft (DFG, German Research Foundation): The first author by Project-ID 507660524, and the second and third by Project-ID 491392403 (TRR 358). The authors would like to thank Sven-Ake Wegner for insightful discussions and contributions.
\end{ack}

\section{Acyclicity and resolving subcategories}

In an abelian category there is a canonical notion of acyclicity of a sequence $L \xrightarrow{f} M \xrightarrow{g} N$; that is when $0 \to \ker(g) \to M \to \coker(f) \to 0$ is a short exact sequence. There are various equivalent characterizations of acyclicity in an abelian category. However, the characterizations can yield different notions in an exact category. The most common notion was introduced in \cite[Section~1]{Neeman:1990}. Our interest lies in a different notion of acyclicity that is motivated by the Ext-kernels introduced in \cite[Definition~3]{Rump:2021}. The latter notion is closely connected to resolving subcategories. 

\subsection{Exact categories and exact functors}

Let $\cat{A}$ be an additive category. A pair of morphisms $(i,p)$ in $\cat{A}$ is called a \emph{kernel-cokernel pair} if $\ker(p) = i$ and $\coker(i) = p$. An \emph{exact structure} is a collection of kernel-cokernel pairs satisfying a certain set of axioms; this was introduced in \cite{Quillen:1973}. An \emph{exact category} $\cat{E}$ is an additive category equipped with an exact structure. Generally, we use $\cat{E}$ to refer to both the underlying additive category and the exact structure. We refer to those kernel-cokernel pairs in the exact structure of $\cat{E}$ as \emph{$\cat{E}$-conflations}. If $(i,p)$ is an $\cat{E}$-conflation, we call $i$ an \emph{$\cat{E}$-inflation} and $p$ an \emph{$\cat{E}$-deflation}. We represent $\cat{E}$-inflations by $\infl[\cat{E}]$ and $\cat{E}$-deflations by $\defl[\cat{E}]$. In other sources inflations and deflations are called admissible monomorphisms and admissible epimorphisms, respectively. For an overview of exact categories see \cite{Buehler:2010}. 

\subsection{Acyclicity}

We recall different notions of acyclicity in an exact category; the first was introduced by \cite[Section~1]{Neeman:1990}.

\begin{definition}
Let $\cat{E}$ be an exact category. A sequence $L \xrightarrow{f} M \xrightarrow{g} N$ with $gf = 0$ is
\begin{enumerate}
\item \emph{$\cat{E}$-acyclic (at $M$)} if there exists a factorization
\begin{equation*}
\begin{tikzcd}[column sep=small]
L \ar[rr,"f"] \ar[dr,defl=\cat{E}] && M \ar[rr,"g"] \ar[dr,defl=\cat{E}] && N \\
& K \ar[ur,infl=\cat{E}] && C \ar[ur,infl=\cat{E}]
\end{tikzcd}
\end{equation*}
where $L \to K$ and $M \to C$ are $\cat{E}$-deflations, $K \to M$ and $C \to M$ are $\cat{E}$-inflations, such that $K \to M \to C$ is an $\cat{E}$-conflation; 
\item \emph{left $\ext{\cat{E}}$-acyclic (at $M$)} if for any $a \colon A \to M$ with $ga = 0$ there exists a commutative diagram
\begin{equation*}
\begin{tikzcd}
B \ar[r,dashed,defl=\cat{E}] \ar[d,dashed] & A \ar[d,"a"] \ar[dr,"0"] \\
L \ar[r,"f"] & M \ar[r,"g"] & N
\end{tikzcd}
\end{equation*}
where $B \to A$ is an $\cat{E}$-deflation;
\item \emph{split acyclic (in $M$)} if it is $\cat{E}^\oplus$-acyclic (in $M$); and
\item \emph{left Hom-acyclic (in $M$)} if it is left $\ext{\cat{E}^{\oplus}}$-acyclic (in $M$).
\end{enumerate}
Here, $\cat{E}^\oplus$ is the exact category with the same underlying additive category as $\cat{E}$ equipped the split exact structure.
\end{definition}

It was shown in \cite[Proposition~9]{Rump:2021} that for an abelian category $\cat{A}$, viewed as an exact category equipped with the abelian exact structure, a sequence $L \xrightarrow{f} M \xrightarrow{g} N$ is $\cat{A}$-acyclic if and only if it is left $\ext{\cat{A}}$-acyclic if and only if $0 \to \ker(g) \to M \to \coker(f) \to 0$ is exact. 

\begin{remark}
The notion of left $\ext{\cat{E}}$-acyclicity is inspired by the Ext-kernels introduced in \cite[Definition~3]{Rump:2021}. In fact, a morphism $L \to M$ is an \emph{$\ext{\cat{E}}$-kernel} of $M \to N$, if and only if $L \to M \to N$ is left $\ext{\cat{E}}$-acyclic. Similarly, a morphism $L \to M$ is a \emph{weak kernel} of $M \to N$, if and only if $L \to M \to N$ is left Hom-acyclic. In particular, the notion of left Hom-acyclicity is closely connected to $n$-exact sequences; see \cite[Definition~2.4]{Jasso:2016}. In \cite[Definition~5]{Rump:2021} a complex over $\cat{E}$ is called left exact if it is left $\ext{\cat{E}}$-acyclic in each degree.
\end{remark}

We immediately observe the following implications from the definition
\begin{equation} \label{connection_diff_acyclic}
\begin{tikzcd}
\text{split acyclic} \ar[r,Rightarrow] \ar[d,Rightarrow] & \text{$\cat{E}$-acyclic} \ar[d,Rightarrow] \\
\text{left Hom-acyclic} \ar[r,Rightarrow] & \text{left $\ext{\cat{E}}$-acyclic} \nospacepunct{.}
\end{tikzcd}
\end{equation}

We collect some properties of left $\ext{\cat{E}}$-acyclic sequences. 

\begin{lemma} \label{mono_ext_acyc_hom_acyc}
Let $\cat{E}$ be an exact category and let $L \xrightarrow{f} M \xrightarrow{g} N$ be sequence in $\cat{E}$ with $gf = 0$. If $f$ is a monomorphism, then the sequence $L \xrightarrow{f} M \xrightarrow{g} N$ is left $\ext{\cat{E}}$-acyclic if and only if it is left Hom-acyclic.

In particular, a sequence $0 \to M \xrightarrow{g} N$ is left $\ext{\cat{E}}$-acyclic if and only if $g$ is a monomorphism.
\end{lemma}
\begin{proof}
By \cref{connection_diff_acyclic} it is enough to show the forward direction. We assume the sequence is left $\ext{\cat{E}}$-acyclic. Let $a \colon A \to M$ be a morphism with $ga = 0$. Then there exists an $\cat{E}$-deflation $p \colon B \to A$ and a morphism $h \colon B \to L$ with $ap = fh$. Let $i \colon K \to B$ be the kernel of $p$. Then $fhi = 0$ and as $f$ is a monomorphism we have $hi = 0$. Hence there exists a morphism $v \colon A \to L$ such that $h = vp$. Hence $ap = fh = fvp$. Since $p$ is an $\cat{E}$-deflation we have $a = fv$, thus $L \xrightarrow{f} M \xrightarrow{g} N$ is left Hom-acyclic.
\end{proof}

\subsection{Exact subcategories}

We recall the terminology for subcategories in an exact category. 

Let $\cat{E}$ be an exact category. An additive subcategory $\cat{C} \subseteq \cat{E}$ is
\begin{enumerate}
\item \emph{extension-closed}, if for any $\cat{E}$-conflation $L \to M \to N$ with $L,N \in \cat{C}$ also $M \in \cat{C}$;
\item \emph{inflation-closed}, if for any $\cat{E}$-conflation $L \to M \to N$ with $L,M \in \cat{C}$ also $N \in \cat{C}$;
\item \emph{deflation-closed}, if for any $\cat{E}$-conflation $L \to M \to N$ with $M,N \in \cat{C}$ also $L \in \cat{C}$;
\item \emph{closed under summands}, if for any $M \oplus M' \in \cat{C}$ also $M,M' \in \cat{C}$; 
\item \emph{thick}, if it is extension-, inflation- and deflation-closed and closed under summands;
\item \emph{Serre}, if for any $\cat{E}$-conflation $L \to M \to N$ one has $M \in \cat{C}$ if and only if $L,M \in \cat{C}$.
\end{enumerate}
Note that a subcategory is deflation-closed if and only if it is closed under kernels of $\cat{E}$-deflations. 

If $\cat{C} \subseteq \cat{E}$ is extension-closed, then there exists an induced exact structure on $\cat{C}$ consisting of the $\cat{E}$-conflations $L \to M \to N$ with $L,M,N \in \cat{C}$. We call this the \emph{fully exact structure (coming from $\cat{E}$)} on $\cat{C}$. 

A functor $\sF \colon \cat{E} \to \cat{F}$ for exact categories $\cat{E}$ and $\cat{F}$ is
\begin{enumerate}
\item \emph{exact}, if it maps $\cat{E}$-conflations to $\cat{F}$-conflations;
\item an \emph{exact embedding}, if it is a faithful exact functor which is injective on objects; 
\item a \emph{fully exact embedding}, if it is an exact embedding that is full and, whenever $\sF(L) \to \tilde{M} \to \sF(N)$ is an $\cat{F}$-conflation with $L,N \in \cat{E}$, there exists $M \in \cat{E}$ such that $\sF(M) = \tilde{M}$ and the $\cat{F}$-conflation is the image of an $\cat{E}$-conflation.
\end{enumerate}
If $\sF$ is an exact embedding that is full, then we can identify $\cat{E}$ with its image in $\cat{F}$ and view $\cat{E}$ as an exact subcategory of $\cat{F}$. Furthermore, when $\sF$ is a fully exact embedding, then the exact structure on $\cat{E}$ consists of all $\cat{F}$-conflations with objects in $\cat{E}$.

An exact category $\cat{E}$ is an \emph{exact subcategory} of an exact category $\cat{F}$, if it is an additive subcategory and the inclusion functor is an exact embedding that is full. When the inclusion functor is a fully exact embedding, we say $\cat{E}$ is a \emph{fully exact subcategory}. In this case, the exact structure on $\cat{E}$ coincides with the fully exact structure coming from $\cat{F}$.

\subsection{Resolving subcategories}

Let $\cat{F}$ be an exact category. A fully exact subcategory $\cat{E} \subseteq \cat{F}$ is \emph{resolving in $\cat{F}$}, if it satisfies
\begin{enumerate}[label=(R\arabic*), ref=R\arabic*]
\item \label{resolving:generator} for any $M \in \cat{F}$, there exists an $\cat{F}$-deflation $L \to M$ with $L \in \cat{E}$; and
\item \label{resolving:deflationclsd} $\cat{E}$ is deflation-closed in $\cat{F}$.
\end{enumerate}
We say $\cat{E} \subseteq \cat{F}$ is \emph{semi-resolving in $\cat{F}$}, if it satisfies
\begin{enumerate}[label=(R\arabic*'), ref=R\arabic*']
\item \label{semiresolving:generator} for any $M \in \cat{F}$ there exists an epimorphism $L \to M$ with $L \in \cat{E}$; and
\item \label{semiresolving:deflationclsd} every epimorphism $M \to N$ in $\cat{F}$ with $M,N \in \cat{E}$ has a kernel which belongs to $\cat{E}$.
\end{enumerate}
If $\cat{F}$ is left abelian, then resolving and semi-resolving are equivalent. However, in general one has \cref{resolving:generator} $\implies$ \cref{semiresolving:generator} and \cref{semiresolving:deflationclsd} $\implies$ \cref{resolving:deflationclsd}.

\begin{remark}
There are many slightly different notions of resolving subcategories. Initially, resolving subcategories were introduced for subcategories of a module category \cite[(3.11)]{Auslander/Bridger:1969}. They were later generalized to abelian categories \cite[Definition~2.1]{Stovicek:2014} and exact categories \cite[Definition~2.3]{Enomoto:2017}. For left abelian categories another notion appears in \cite[Definition~2]{Rump:2020} and for deflation-exact categories in \cite[Definition~3.1]{Henrard/vanRoosmalen:2020c}; these are not generalizations of Enomoto's definition. We use the definition from \cite[Definition~3.1]{Henrard/vanRoosmalen:2020c}, and semi-resolving is the notion from \cite[Definition~2]{Rump:2020}.
\end{remark}

Let $\cat{E} \subseteq \cat{F}$ be a resolving subcategory. One can deduce from the axioms \cref{resolving:generator,resolving:deflationclsd} that for each $M \in \cat{F}$, there exists an $\cat{F}$-acyclic sequence 
\begin{equation} \label{eq:resolving_resn}
\dots \to X^{n-1} \to X^{n} \to \dots \to X^{-1} \to X^{0} \to M \to 0
\end{equation}
with $X^{i} \in \cat{E}$ for all $i \leq 0$. 

A resolving subcategory $\cat{E} \subseteq \cat{F}$ is called \emph{finitely resolving}, if for every $M \in \cat{F}$ there exists an integer $n \leq 0$ and a sequence as in \cref{eq:resolving_resn} with $X^i = 0$ for $i < n$.

For resolving subcategories the notion of Ext-acyclicity is preserved in the following sense:

\begin{lemma} \label{resolving_ext_acyc}
Let $\cat{F}$ be an exact category and $\cat{E} \subseteq \cat{F}$ a resolving subcategory. A sequence $L \xrightarrow{f} M \xrightarrow{g} N$ with $L,M,N \in \cat{E}$ is left $\ext{\cat{F}}$-acyclic if and only if it is left $\ext{\cat{E}}$-acyclic.
\end{lemma}
\begin{proof}
For the forward direction, let $a \colon A \to M$ be a morphism in $\cat{E}$ with $g a = 0$. As the sequence is left $\ext{\cat{F}}$-acyclic, we have the following commutative diagram
\[
\begin{tikzcd}[column sep=large]
B \ar[r,"p" {near start},defl=\cat{F}] \ar[d,"h"] & A \ar[d,"a"] \ar[dr,"0"] & \\
L \ar[r,"f"] & M \ar[r,"g"] & N
\end{tikzcd}
\]
where $p$ is an $\cat{F}$-deflation. By \cref{resolving:generator}, there exists an $\cat{F}$-deflation $q \colon C \to B$ with $C \in \cat{E}$. By \cref{resolving:deflationclsd}, the composition $q p$ is an $\cat{E}$-deflation and $apq = fhq$. 

For the converse, we assume the sequence $L \xrightarrow{f} M \xrightarrow{g} N$ is left $\ext{\cat{E}}$-acyclic. Let $a \colon A \to M$ be a morphism in $\cat{F}$ with $ga = 0$. By \cref{resolving:generator}, there exists an $\cat{F}$-deflation $q \colon \hat{A} \to A$ with $\hat{A} \in \cat{E}$. Then $aq$ is a morphism in $\cat{E}$ with $gaq = 0$. Hence there exists an $\cat{E}$-deflation $p \colon B \to \hat{A}$ and a morphism $h \colon B \to L$ such that $aqp = fh$ ane $qp$ is an $\cat{F}$-deflation. 
\end{proof}

As a direct consequence of this \namecref{resolving_ext_acyc}, any $\cat{F}$-acyclic sequence with objects in $\cat{E}$ is left $\ext{\cat{E}}$-acyclic. We will construct ambient categories of $\cat{E}$ in which the converse holds; see \cref{ext_acyc_in_loc_acyc}. 

\section{Ext resolutions}

The sequence in \cref{eq:resolving_resn} obtained from a resolving subcategory plays an important role in the sequel. In fact, one can think of $\cdots \to X^{-1} \to X^{-0}$ as a resolution of $M$. In view of \cref{resolving_ext_acyc} we are interested in such sequences that are left Ext-acyclic in negative degrees. 

Let $\cat{E}$ be an exact category. We denote the category of complexes over $\cat{E}$ by $\ccat{\cat{E}}$. We use cohomological notation for complexes; that is we write a complex $X$ as
\begin{equation*}
\cdots \to X^{n-1} \xrightarrow{d_X^{n-1}} X^{n} \xrightarrow{d_X^{n}} X^{n+1} \to \cdots\,.
\end{equation*}
A complex $X$ is \emph{$\cat{E}$-acyclic}, if it is $\cat{E}$-acyclic in each degree. A morphism of complexes is an \emph{$\cat{E}$-quasi-isomorphism}, if its mapping cone is $\cat{E}$-acyclic. 

\begin{definition}
Let $\cat{E}$ be an exact category. An \emph{$\ext{\cat{E}}$-resolution} is a complex $X$ concentrated in non-negative degrees that is left $\ext{\cat{E}}$-acyclic in negative degrees. We call an $\ext{\cat{E}}$-resolution $X$ \emph{bounded}, if $X^{n} = 0$ for $n \ll 0$. 
\end{definition}

In this definition we do not require that the Ext-resolution `resolves' any object, as the object resolved needs not lie in $\cat{E}$. We still use the term `resolution' as it will be a resolution of an object in the ambient categories of $\cat{E}$ constructed in \cref{sec:heart,sec:rclosure}.

\subsection{Gluing Ext-acyclic complexes}

We will show that Ext-acyclicity is compatible with the formation of the mapping cone. In particular, it behaves well in the homotopy category; see \cref{sec:dcat}.

\begin{lemma} \label{ConeExtacyclic} 
Let $\cat{E}$ be an exact category. Let $f \colon X \to Y$ be a morphism in $\ccat{\cat{E}}$. If $X$ is left $\ext{\cat{E}}$-acyclic in degree $n+1$ and $Y$ is left $\ext{\cat{E}}$-acyclic in degree $n$, then $\cone(f)$ is left $\ext{\cat{E}}$-acyclic in degree $n$.
\end{lemma}

\begin{proof}
Let $g \colon A \xrightarrow{\begin{psmallmatrix} a \\ b \end{psmallmatrix}} X^{n+1} \oplus Y^{n}$ be a morphism such that the diagram
\begin{equation*}
\begin{tikzcd}[ampersand replacement=\&,column sep=huge]
\&[-4em] \& A \ar[d,"{\begin{psmallmatrix} a \\ b \end{psmallmatrix}}"] \ar[dr,"0"] \\
\cdots \ar[r] \& X^{n} \oplus Y^{n-1} \ar[r,"{\begin{psmallmatrix} -d_X^{n} & 0 \\
f^{n} & d_Y^{n-1} \end{psmallmatrix}}" swap] \& X^{n+1} \oplus Y^{n} \ar[r,"{\begin{psmallmatrix}-d_X^{n+1} & 0 \\
f^{n+1} & d_Y^{n}\end{psmallmatrix}}" swap] \& X^{n+2}\oplus Y^{n+1} \ar[r] \&[-4em] \cdots
\end{tikzcd}
\end{equation*}
commutes. The bottom row is $\cone(f)$ in degrees $n-1$, $n$ and $n+1$. 

As $d_X^{n+1} a=0$ and $X$ left $\ext{\cat{E}}$-acyclic in degree $n+1$, there exists an $\cat{E}$-deflation $p\colon B \to A$ and a morphism $h \colon B \to X^{n}$ such that the following diagram 
commutes 
\[
\begin{tikzcd}[column sep=large]
B \arrow[r,"p" {near start},defl=\cat{E}] \arrow[d, "h"'] & A \arrow[d, "a"] \ar[dr,"0"] \\
X^{n} \arrow[r, "d_X^{n}"] & X^{n+1} \ar[r,"d_X^{n+1}"] & X^{n+2} \nospacepunct{.}
\end{tikzcd}
\]
By assumption, we have $f^{n+1} a + d_Y^{n} b = 0$. This yields
\[
0 = (f^{n+1} a + d_Y^{n} b) p = f^{n+1} a p + d_Y^{n} b p = f^{n+1} d_X^{n} h + d_Y^{n} b p = d_Y^{n} (f^{n} h + b p)\,.
\]
As $Y$ is left $\ext{\cat{E}}$-acyclic in degree $n$, we find an $\cat{E}$-deflation $q\colon C \to B$ and a morphism $g \colon C \to Y^{n-1}$ such that the following diagram commutes 
\[
\begin{tikzcd}[column sep=large]
C \arrow[r,"q" {near start},defl=\cat{E}] \arrow[d, "g"'] & B \arrow[d, "{f^n h + b p}" description] \ar[dr,"0"] \\
Y^{n-1} \arrow[r, "d_Y^{n-1}"] & Y^{n} \ar[r,"d_Y^{n}"] & Y^{n+1} \nospacepunct{.}
\end{tikzcd}
\]
Combining the identities yields the commutative square
\[
\begin{tikzcd}[ampersand replacement=\&]
C \arrow[d,"{\begin{psmallmatrix}-hq \\ g \end{psmallmatrix}}"] \arrow[r,"q" {near start},defl=\cat{E}] \& B \ar[r,"p" {near start},defl=\cat{E}] \& A \arrow[d, "{\begin{psmallmatrix} a \\ b \end{psmallmatrix}}"] \ar[dr,"0"] \&[+3em] \\
X^{n}\oplus Y^{n-1} \ar[rr,"{\begin{psmallmatrix} -d_X^{n} & 0 \\ f^{n} & d_Y^{n-1} \end{psmallmatrix}}" swap] \&\& X^{n+1} \oplus Y^{n} \ar[r,"{\begin{psmallmatrix} -d_X^{n+1} & 0 \\ f^{n+1} & d_Y^{n} \end{psmallmatrix}}" swap] \& X^{n+2} \oplus Y^{n+1}
\end{tikzcd}
\]
where the composition of the top row $pq$ is an $\cat{E}$-deflation.
\end{proof}

\begin{lemma} \label{Extacyc_htpc_retract}
Let $\cat{E}$ be an exact category. Let $f \colon X \to Y$ be a homotopy retract in $\ccat{\cat{E}}$; that is, there exists a chain map $g \colon Y \to X$ such that $gf$ is homotopic to the identity map on $X$. If $Y$ is left $\ext{\cat{E}}$-acyclic in degree $n$, then so is $X$.
\end{lemma}
\begin{proof}
Let $f$ and $g$ be as in the statement, and let $h$ be the homotopy witnessing that $gf$ is homotopic to $\id_X$. Let $a \colon A \to X^n$ be a morphism such that $d_X^n a = 0$. Then 
\begin{equation*}
d_Y^n f^n a = f^{n+1} d_X^n a = 0\,.
\end{equation*}
As $Y$ is left $\ext{\cat{E}}$-acyclic in degree $n$, there exists an $\cat{E}$-deflation $p \colon B \to A$ and a morphism $u \colon B \to Y^{n-1}$ such that $f^n a p = d_Y^{n-1} u$. Now we have
\begin{equation*}
\begin{aligned}
ap &= \id_{X^n} a p = (g^n f^n + d_X^{n-1} h^n + h^{n+1} d_X^n) a p \\
&= g^n d_Y^{n-1} u + d_X^{n-1} h^n a p = d_X^{n-1} (g^{n-1} u + h^n a p)\,.
\end{aligned}
\end{equation*}
Hence $X$ is left $\ext{\cat{E}}$-acyclic in degree $n$.
\end{proof}

\begin{corollary} \label{Extacyc_cone_then}
Let $\cat{E}$ be an exact category. Let $f \colon X \to Y$ be a morphism in $\ccat{\cat{E}}$. 
\begin{enumerate}
\item If $X$ and $\cone(f)$ are left $\ext{\cat{E}}$-acylic in degree $n$, then so is $Y$.
\item If $Y$ is left $\ext{\cat{E}}$-acyclic in degree $n$ and $\cone(f)$ is left $\ext{\cat{E}}$-acyclic in degree $n-1$, then $X$ is left $\ext{\cat{E}}$-acyclic in degree $n$.
\end{enumerate}
\end{corollary}
\begin{proof}
The cylinder and cocylinder of $f \colon X \to Y$ are
\begin{equation*}
\begin{aligned}
\cyl(f) &\colonequals \cone(\susp^{-1} \cone(f) \to X) \quad \text{and} \\
\cocyl(f) &\colonequals \cone(\susp^{-1} Y \to \susp^{-1} \cone(f))\,,
\end{aligned}
\end{equation*}
respectively. It is well-known that there are homotopy equivalences $Y \to \cyl(f)$ and $X \to \cocyl(f)$. Now the claim follows from \cref{ConeExtacyclic,Extacyc_htpc_retract}.
\end{proof}

\subsection{Morphisms between Ext-resolutions} \label{sec:mor_ext_rsn}

Ext-resolutions have weaker properties than classical resolutions. In fact, inducing a morphism, and the uniqueness of these morphisms, holds up to a quasi-isomorphism.

\begin{lemma} \label{induced_mor_extresn}
Let $\cat{E}$ be an exact category and let $X$ be a complex over $\cat{E}$ concentrated in non-positive degrees and $Y$ an $\ext{\cat{E}}$-resolution. A commutative diagram
\begin{equation*}
\begin{tikzcd}
\cdots \ar[r] & X^{-2} \ar[r,"d_X^{-2}"] & X^{-1} \ar[r,"d_X^{-1}"] \ar[d,"f^{-1}"] & X^0 \ar[d,"f^0"] \ar[r] & 0 \ar[r] \ar[d,"f^1=0"] & \cdots \\
\cdots \ar[r] & Y^{-2} \ar[r,"d_Y^{-2}"] & Y^{-1} \ar[r,"d_Y^{-1}"] & Y^0 \ar[r] & 0 \ar[r] & \cdots
\end{tikzcd}
\end{equation*}
can be, up to $\cat{E}$-quasi-isomorphism, be extended to a morphism of complexes. Explicitly, there exists a complex $W$ concentrated in non-positive degrees, an $\cat{E}$-quasi-isomorphism $g \colon W \to X$ and a chain map $\hat{f} \colon W \to Y$, with $g^n = \id$ and $\hat{f}^n = f^n$ for $n \geq -1$. 
\end{lemma}
\begin{proof}
We use induction on the cohomological degree to construct the complex $W$, the $\cat{E}$-quasi-isomorphism $g \colon W \to X$ and the chain map $\hat{f} \colon W \to Y$. 

Set $W^n = X^n$ and $d_W^n = d_X^n$ and $g^n = \id_{X^n}$ for $n \geq -1$. Further, set $\hat{f}^n = f^n$ for $n = \geq -1$.

We claim, that for every $n$ there exists a commutative diagram
\begin{equation} \label{ext_acyc_extend_to_chainmap:ind}
\begin{tikzcd}[column sep=large]
X^{n-1} \ar[d,"{d_X^{n-1}}" swap] & & & Y^{n-1} \ar[d,"{d_Y^{n-1}}"] \\
X^{n} \ar[d,"{d_X^{n}}" swap] \ar[dr,"{a^{n}}"] & P^{n} \ar[l,defl=\cat{E},"{q^{n}}" {swap,near start}] \ar[dr,"{b^{n}}"] & W^{n} \ar[d,"{d_W^{n}}"] \ar[r,"{\hat{f}^{n}}"] \ar[l,defl=\cat{E},"{p   ^{n}}"  {swap,near start}] & Y^{n} \ar[d,"{d_Y^{n}}"] \\
X^{n+1} & P^{n+1} \ar[l,defl=\cat{E},"{q^{n+1}}" {near start}] & W^{n+1} \ar[r,"{\hat{f}^{n+1}}" swap] \ar[l,defl=\cat{E},"{p^{n+1}}" {near start}] & Y^{n+1}
\end{tikzcd}
\end{equation}
where $P^{n}$ is the pullback of the span $X^{n} \to P^{n+1} \to W^{n+1}$ and $a^n d_X^{n-1} = 0$. 

For $n \geq -1$, we set $P^n = X^n$, $a^{n} = b^{n} = d_X^{n}$ and take the remaining morphisms to be the identity morphism. Clearly, the diagram \cref{ext_acyc_extend_to_chainmap:ind} commutes and the square with corners $P^{n+1}$, $W^{n+1}$, $X^{n}$ and $P^{n}$ is a pullback square and $a^n d_X^{n-1} = 0$. 

We assume the morphism and objects in \cref{ext_acyc_extend_to_chainmap:ind} exist with the desired properties for some $n \leq -1$. Since $P^{n}$ is a pullback and $a^n d_X^{n-1} = 0$, there exists a morphism $a^{n-1} \colon X^{n-1} \to P^{n}$ such that
\begin{equation*}
d_X^{n-1} = q^{n} a^{n-1} \quad \text{and} \quad b^{n} a^{n-1} = 0\,.
\end{equation*}
As $q^n a^{n-1} d_X^{n-2} = 0$ and $b^n a^{n-1} d_X^{n-2} = 0$, the universal property of the pullback yields $a^{n-1} d_X^{n-2} = 0$. 
Let $P^{n-1}$ be the pullback of $a^{n-1}$ along $p^n$ with canonical morphisms $q^{n-1} \colon P^{n-1} \to X^{n-1}$ and $b^{n-1} \colon P^{n-1} \to W^{n}$. We obtain
\begin{equation*}
\begin{aligned}
d_Y^{n} \hat{f}^{n} b^{n-1} &= \hat{f}^{n+1} d_W^{n} b^{n-1} = \hat{f}^{n+1} b^{n} p^{n} b^{n-1} \\
&= \hat{f}^{n+1} b^{n} a^{n-1} q^{n-1} = 0\,.
\end{aligned}
\end{equation*}
Since $Y$ is left $\ext{\cat{E}}$-acyclic in degree $n$, there exists a commutative diagram
\begin{equation*}
\begin{tikzcd}[column sep=large]
P^{n-1} \ar[dr,"{b^{n-1}}" swap] & W^{n-1} \ar[l,defl=\cat{E},"{p^{n-1}}" , swap] \ar[r,"\hat{f}^{n-1}"] & Y^{n-1} \ar[d,"{d_Y^{n-1}}"] \\
& W^{n} \ar[r,"\hat{f}^{n}"] & Y^{n} \nospacepunct{.}
\end{tikzcd}
\end{equation*}
Setting $d_W^{n-1} \colonequals b^{n-1} p^{n-1}$ yields the diagram \cref{ext_acyc_extend_to_chainmap:ind} with the desired properties for $n-1$.

From the construction we obtain
\begin{equation*}
d_W^{n} d_W^{n-1} = b^{n} p^{n} b^{n-1} p^{n-1} = b^{n} a^{n-1} q^{n-1} p^{n-1} = 0
\end{equation*}
and hence $W$ is a complex. From the commutativity of \cref{ext_acyc_extend_to_chainmap:ind} we have that $\hat{f} \colon W \to Y$ is a chain map and $g^{n} \colonequals q^{n} p^{n}$ defines a chain map $g \colon W \to X$ as well. It remains to observe that the sequences
\begin{equation*}
P^{n} \xrightarrow{\begin{psmallmatrix} b^{n} \\ q^{n} \end{psmallmatrix}} W^{n+1} \oplus X^{n} \xrightarrow{\begin{psmallmatrix} -p^{n+1} & a^{n} \end{psmallmatrix}} P^{n+1}
\end{equation*}
are $\cat{E}$-conflations for all $n$. Hence $\cone(g)$ is $\cat{E}$-acyclic in every degree and $g$ is an $\cat{E}$-quasi-isomorphism. By construction $W$ is concentrated in non-positive degrees.
\end{proof}

\begin{lemma} \label{ind_mor_nullhptc}
Let $\cat{E}$ be an exact category and $f \colon X \to Y$ a morphism in $\ccat{\cat{E}}$ with $X$ and $Y$ $\ext{\cat{E}}$-resolutions. If there exists a morphism $h^{0} \colon X^0 \to Y^{-1}$ such that $f^0 = h^0 d_Y^{-1}$, then there exists an $\cat{E}$-quasi-isomorphism $g \colon W \to X$ such that $fg$ is null-homotopic.
\end{lemma}
\begin{proof}
We use induction on the cohomological degree to construct the complex $W$, the $\cat{E}$-quasi-isomorphism $g \colon W \to X$ and the map $h \colon W \to \susp^{-1} Y$ such that $h$ is the homotopy that witnesses that $fg$ is null-homotopic.

Set $W^{n} = X^{n}$ for $n \geq 0$. We claim, that for every $n$, there exists a commutative diagram
\begin{equation} \label{ind_mor_nullhptc:IV}
\begin{tikzcd}[row sep=large,column sep=large]
& & X^{n-1} \ar[r,"f^{n-1}"] \ar[d,"d_X^{n-1}"] & Y^{n-1} \ar[d,"d_Y^{n-1}"] \\
W^{n} \ar[r,defl=\cat{E},"p^{n}" {near start}] \ar[d,"d_W^{n}"] & P^{n} \ar[r,defl=\cat{E},"q^{n}" {near start}] \ar[dl,"b^{n}" swap] & X^{n} \ar[r,"f^{n}"] \ar[d,"d_X^{n}"] \ar[dl,"a^{n}" swap] & Y^{n} \ar[d,"d_Y^{n}"] \\
W^{n+1} \ar[r,defl=\cat{E},"p^{n+1}"] & P^{n+1} \ar[r,defl=\cat{E},"q^{n+1}"] & X^{n+1} \ar[r,"f^{n+1}"] & Y^{n+1}
\end{tikzcd}
\end{equation}
and morphisms $h^{n+1} \colon W^{n+1} \to Y^{n}$ and $h^{n} \colon W^{n} \to Y^{n-1}$ such that $P^{n}$ is the pullback of the span $X^{n}\rightarrow P^{n+1}\to W^{n+1}$ and 
\begin{equation*}
f^{n} q^{n} p^{n} = d^{n-1}_Y h^{n} + h^{n+1} d^{n}_W \quad \text{and} \quad a^n d_X^{n-1} = 0\,.
\end{equation*}

For $n \geq 0$ we set $P^{n} = X^{n}$ and $a^{n} = b^{n} = d_X^{n}$. When $n = 0$, the morphism $h^0$ is part of the assumptions and we further set $h^n = 0$ for $n > 0$. With these assignments the claim holds for $n \geq 0$.

We assume the morphisms and objects in \cref{ind_mor_nullhptc:IV} exist with the desired properties for some $n \leq 0$. Since $P^{n}$ is a pullback, there exists a morphism $a^{n-1} \colon X^{n-1} \to P^{n}$ such that
\begin{equation*}
d_X^{n-1} = q^{n} a^{n-1} \quad \text{and} \quad 0 = b^{n} a^{n-1}\,.
\end{equation*}
As $q^{n} b^{n} d_X^{n-2} = 0$ and $b^{n} a^{n-1} d_X^{n-2} = 0$, the universal property of the pullback yields $a^{n-1} d_X^{n-2} = 0$. Let $P^{n-1}$ be the pullback of $a^{n-1}$ along $p^{n}$ with canonical morphisms $q^{n-1} \colon P^{n-1} \to X^{n-1}$ and $b^{n-1} \colon P^{n-1} \to W^{n}$. We obtain
\begin{equation*}
\begin{aligned}
&\quad d_Y^{n-1} (f^{n-1} q^{n-1} - h^{n} b^{n-1}) \\
&= f^{n} d_X^{n-1} q^{n-1} - f^{n} q^{n} p^{n} b^{n-1} + h^{n+1} d_W^{n} b^{n-1} = 0\,.
\end{aligned}
\end{equation*}
Since $Y$ is left $\ext{\cat{E}}$-acyclic in degree $n-1$ for $n \leq 0$, there exists an object $W^{n-1}$, a morphism $h^{n-1} \colon W^{n-1} \to Y^{n-2}$ and an $\cat{E}$-deflation $p^{n-1} \colon W^{n-1} \to P^{n-1}$ such that
\begin{equation*}
f^{n-1} q^{n-1} p^{n-1} = h^{n} b^{n-1} p^{n-1} + d_Y^{n-2} h^{n-1}\,.
\end{equation*}
It remains to set $d_W^{n-1} \colonequals b^{n-1} p^{n-1}$. 

By construction $W$ is a complex, and $g \colon W \to X$ given by $g^{n} \colonequals q^{n} p^{n}$ is a chain map. It remains to observe that the sequences
\begin{equation*}
P^{n} \to X^{n} \oplus W^{n+1} \to P^{n+1}
\end{equation*}
are $\cat{E}$-conflations. Hence the complex $\cone(g)$ is $\cat{E}$-acyclic in every degree, $g$ is an $\cat{E}$-quasi-isomorphism and $fg$ is null-homotopic.
\end{proof}

\begin{remark}
The induced morphism in \cref{induced_mor_extresn} need not be unique. However, given two lifts $X \xleftarrow{\sim} W \to Y$ and $X \xleftarrow{\sim} W' \to Y$, there exist quasi-isomorphisms $V \xrightarrow{\sim} W$ and $V \xrightarrow{\sim} W'$ such that the compositions to $X$ coincide. Hence \cref{ind_mor_nullhptc} can be applied to the difference of $V \to W \to Y$ and $V \to W' \to Y$. This yields another quasi-isomorphism and a null-homotopic morphism. In this sense the induced morphism from \cref{induced_mor_extresn} is unique.
\end{remark}

\section{In the derived category} \label{sec:dcat}

In the previous section we saw that left Ext-acyclicty behaves well with respect to many operations of complexes. In particular, we can consider left Ext-acyclicty in the setting of the homotopy and derived category. We use it to define an analogue of the standard t-structure in the derived category of an abelian category. While the resulting pair need not be a t-structure, their `heart' is an interesting extension of the initial exact category.

Let $\cat{E}$ be an exact category with underlying additive category $\cat{A}$. We denote by $\kcat{\cat{A}}$ the homotopy category of $\cat{A}$ and by $\dcat{\cat{E}}$ the derived category of $\cat{E}$; recall that the derived category $\dcat{\cat{E}} = \kcat{\cat{A}}/\operatorname{thick}(\Ac{\cat{E}})$ is the Verdier quotient of the homotopy category by the thick subcategory generated by the $\cat{E}$-acyclic complexes; see \cite[Construction~1.5]{Neeman:1990}. When $\cat{E}$ is (weakly) idempotent complete, then the $\cat{E}$-acyclic complexes already form a thick subcategory. We write $\cat{A}$ to emphasize that the formation of the homotopy category depends only on the additive, and not the exact, structure.

We consider different boundedness conditions for the homotopy and derived category: We write $\dbcat{\cat{E}}$ for the derived category of bounded complexes, $\dcat[-]{\cat{E}}$ the derived category of right bounded complexes and $\dcat[+]{\cat{E}}$ the derived category of left bounded complexes. 

\begin{corollary} \label{ext_acyc_kcat_dcat}
Let $\cat{E}$ be an exact category with underlying additive category $\cat{A}$. For $* \in \{\emptyset,b,+,-\}$ the following holds:
\begin{enumerate}
\item \label{ext_acyc_kcat_dcat:kcat} The subcategory of $\kcat[*]{\cat{A}}$ of complexes that are left $\ext{\cat{E}}$-acylic in degree $n$ is closed under direct summands, isomorphisms and extensions. 
\item \label{ext_acyc_kcat_dcat:dcat} The subcategory of $\dcat[*]{\cat{E}}$ of complexes that are left $\ext{\cat{E}}$-acylic in degree $n$ is closed under direct summands, isomorphisms and extensions. 
\end{enumerate}
\end{corollary}

\begin{proof}
The category in \cref{ext_acyc_kcat_dcat:kcat} is closed under direct summands and isomorphisms by \cref{Extacyc_htpc_retract}, and it is extension-closed by \cref{ConeExtacyclic}. As $\cat{E}$-acyclicity implies left Ext$_{\cat{E}}$-acylicity by \cref{connection_diff_acyclic}, the subcategory in \cref{ext_acyc_kcat_dcat:kcat} contains all $\cat{E}$-acyclic complexes. Hence \cref{ext_acyc_kcat_dcat:kcat} implies \cref{ext_acyc_kcat_dcat:dcat}. 
\end{proof}

\subsection{Candidate for a t-structure} \label{sec:tpair}

The derived category of an abelian category has a standard t-structure whose heart is the initial abelian category. On the derived category of an exact category such a t-structure need not exist, as the heart of any t-structure is abelian. With additional assumptions on the exact category, it is possible to construct a t-structure on the derived category of an exact category; see \cite[Section~2.1]{Huber:1995}, \cite[Section~1.2.2]{Schneiders:1999} and \cite[Proposition~3.5]{Henrard/Kvamme/vanRoosmalen/Wegner:2023}. These constructions all require the existence of kernels. Using left Ext-acyclicity, we define a pair of subcategories that generalizes these constructions. We will give necessary and sufficient conditions for this pair to be a t-structure. 

Let $\cat{E}$ be an exact category and $* \in \{\emptyset,b,+,-\}$. We denote by $\cat{U}^*(\cat{E})$ and $\cat{V}^*(\cat{E})$ the full subcategories of $\dcat[*]{\cat{E}}$ containing the complexes that are $\cat{E}$-acyclic in positive and negative degrees, respectively. For an abelian category $\cat{A}$, the pair $(\cat{U}(\cat{A}),\cat{V}(\cat{A}))$ is the standard t-structure. However, in general $\cat{U}(\cat{E}) \cap \cat{V}(\cat{E})$ is not necessarily abelian. 

We denote by $\cat{V}^*_\ell(\cat{E})$ and $\cat{U}^*_r(\cat{E})$ the full subcategories of $\dcat[*]{\cat{E}}$ containing the complexes that are left $\ext{\cat{E}}$-acyclic in negative degrees and that are right $\ext{\cat{E}}$-acyclic in positive degrees, respectively. From \cref{connection_diff_acyclic} we have
\begin{equation*}
\cat{U}^*(\cat{E}) \subseteq \cat{U}^*_r(\cat{E}) \quad \text{and} \quad \cat{V}^*(\cat{E}) \subseteq \cat{V}^*_\ell(\cat{E})\,.
\end{equation*}
Also note, that $(\cat{U}^*(\cat{E}))^\op = \cat{V}^*(\cat{E}^\op)$ and $(\cat{U}^*_r(\cat{E}))^\op = \cat{V}^*_\ell(\cat{E}^\op)$. In the following we consider the pair $(\cat{U}^*(\cat{E}),\cat{V}^*_\ell(\cat{E}))$, the dual results hold for $(\cat{U}^*_r(\cat{E}),\cat{V}^*(\cat{E}))$. 

\begin{proposition} \label{ext_acyc_t_pair}
Let $\cat{E}$ be an exact category. For $* \in \{\emptyset,b,+,-\}$, the pair $(\cat{U}^*(\cat{E}),\cat{V}^*_\ell(\cat{E}))$ satisfies the following conditions:
\begin{enumerate}
\item \label{ext_acyc_t_pair:strict} $\cat{U}^*(\cat{E})$ and $\cat{V}^*_\ell(\cat{E})$ are closed under direct summands, isomorphisms and extensions;
\item \label{ext_acyc_t_pair:susp} $\susp \cat{U}^*(\cat{E}) \subseteq \cat{U}^*(\cat{E})$ and $\cat{V}^*_\ell(\cat{E}) \subseteq \susp \cat{V}^*_\ell(\cat{E})$; and
\item \label{ext_acyc_t_pair:Hom} $(\susp \cat{U}^*(\cat{E}))^\perp = \cat{V}^*_\ell(\cat{E})$ in $\dcat[*]{\cat{E}}$.
\end{enumerate}
\end{proposition}

\begin{proof}
Part \cref{ext_acyc_t_pair:strict} is a direct consequence of \cref{ext_acyc_kcat_dcat}. Part \cref{ext_acyc_t_pair:susp} holds by construction. It remains to show \cref{ext_acyc_t_pair:Hom}. 

We first show $\Hom_{\dcat[*]{\cat{E}}}(\susp \cat{U}^*(\cat{E}),\cat{V}^*_\ell(\cat{E}))=0$. Let $X \in \susp \cat{U}^*(\cat{E})$, $Y \in \cat{V}^*_\ell(\cat{E})$, and let $f \colon X \to Y$ a morphism in $\dcat[*]{\cat{E}}$. Without loss of generality we can assume that $f$ is a morphism of complexes, as $\cat{U}^*(\cat{E})$ and $\cat{V}^*_\ell(\cat{E})$ are closed under quasi-isomorphisms. Since $X$ is $\cat{E}$-acyclic in non-negative degrees, there is an $\cat{E}$-quasi-isomorphism
\begin{equation*}
\begin{tikzcd}
\cdots \ar[r] & X^{-2} \ar[r,"{d_X^{-2}}"] \ar[d,"="] & \ker(d_X^{-1}) \ar[r] \ar[d] & 0 \ar[r] \ar[d] & \cdots \\
\cdots \ar[r] & X^{-2} \ar[r,"{d_X^{-2}}"] & X^{-1} \ar[r,"{d_X^{-1}}"] & X^0 \ar[r] & \cdots \nospacepunct{.}
\end{tikzcd}
\end{equation*}
Hence we may assume that $X$ is concentrated in negative degrees. We apply \cref{ind_mor_nullhptc} to the brutal truncation $f^{\leqslant 0} \colon X = X^{\leqslant 0} \to Y^{\leqslant 0}$ to obtain an $\cat{E}$-quasi-isomorphism $g \colon W \to X$ such that $f^{\leqslant 0} g$ is null-homotopic. Hence $fg$ is null-homotopic and $f$ is zero in $\dcat{\cat{E}}$. 

Let $X \in (\susp \cat{U}^*(\cat{E}))^\perp$ in $\dcat[*]{\cat{E}}$. For $n \leq -1$, let $a \colon A \to X^{n}$ be a morphism with $d_X^{n} a = 0$. We can view $a$ as a chain map $\hat{a} \colon \susp^{-n} A \to X$. As $\susp^{-n} A \in \susp \cat{U}^*(\cat{E})$ for $n \leq -1$, the morphism $\hat{a}$ is zero in $\dcat[*]{\cat{E}}$, meaning that there exists an $\cat{E}$-quasi-isomorphism $g \colon Y \to \susp^{-n} A$ such that $\hat{a} g$ is null-homotopic. The complex $W \colonequals \cone(g)$ is $\cat{E}$-acyclic and we let $j \colon \susp^{-n} A \to W$ be the canonical morphism. As the homotopy category is triangulated, there exists a morphism $h \colon W \to X$ such that $hj$ is homotopic to $\hat{a}$; this means there exists a morphism $\sigma \colon A \to X^{n-1}$ such that $\hat{a}^{n} - h^{n} j^{n} = d_X^{n-1} \sigma$. 

We consider the following diagram
\[
\begin{tikzcd}[column sep=large]
& & A \ar[d, "j^{n}"] \ar[rd,"0"] & \\
\cdots \ar[r] & W^{n-1} \ar[d] \ar[r, "d^{n-1}_W"] & W^{n} \ar[d,"h^{n}"] \ar[r,"d_W^{n}"] & W^{n+1} \ar[d] \ar[r] & \cdots \\
\cdots \ar[r] & X^{n-1} \ar[r,"d^{n-1}_X"] & X^{n} \ar[r, "d^{n}_X"] & X^{n+1} \ar[r] & \cdots \nospacepunct{.}
\end{tikzcd}
\]
As $W$ is $\cat{E}$-acyclic in every degree, the differential $d_W^{n-1}$ factors as $i p$ with $p$ an $\cat{E}$-deflation and $i = \ker(d_W^{n}) \colon K \to W^{n}$ an $\cat{E}$-inflation. As $d_W^{n} j^{n} = 0$ there is a unique morphism $u \colon A \to K$ such that $j^{n} = i u$. As $p$ is an $\cat{E}$-deflation, we take the pullback of $p$ along $u$. This gives an $\cat{E}$-deflation $q \colon B \to A$ and a morphism $v \colon B \to W^{n-1}$ such that $uq = pv$. This yields
\begin{equation*}
h^{n} j^{n} q = h^{n} i u q = h^{n} i p v = h^{n} d_W^{n-1} v = d_X^{n-1} h^{n-1} v\,.
\end{equation*}
Combining this with the fact that $hj$ is homotopic to $\hat{a}$ yields
\begin{equation*}
aq = \hat{a}^{n} q = (h^{n} j^{n} + d_X^{n-1} \sigma) q = d_X^{n-1} (h^{n-1} v + \sigma q)\,.
\end{equation*}
Since this holds for any $a$ and any $n \geq 1$, we have shown that $X$ is left $\ext{\cat{E}}$-acyclic in negative degrees. Hence $X \in \cat{V}^*_\ell(\cat{E})$. 
\end{proof}

\begin{remark}
We do not expect ${}^\perp \cat{V}^*_\ell(\cat{E}) = \susp \cat{U}^*(\cat{E})$ to hold in general. Although it is not clear which complexes are contained in ${}^\perp \cat{V}^*_\ell(\cat{E})$ besides $\susp \cat{U}^*(\cat{E})$, the subcategory $\susp \cat{U}^*_r(\cat{E})$ is not contained in ${}^\perp \cat{V}^*_\ell(\cat{E})$ in general: Let $\cat{E}$ be a weakly idempotent complete exact category and $L \to M \to N$ a kernel-cokernel pair in $\cat{E}$. The complex
\begin{equation*}
X \colonequals (\cdots \to 0 \to L \to M \to N \to 0 \to \cdots)
\end{equation*}
with $M$ in degree 0 is left $\ext{\cat{E}}$-acyclic in non-positive degrees and right $\ext{\cat{E}}$-acyclic in non-negative degrees; that is $X \in \cat{U}^*_r(\cat{E}) \cap \cat{V}^*_\ell(\cat{E})$. However, the identity morphism $\id_X$ is zero in $\dcat[*]{\cat{E}}$ if and only if $L \to M \to N$ is an $\cat{E}$-conflation.
\end{remark}

Finally, we provide conditions when $(\cat{U}^*(\cat{E}),\cat{V}^*_\ell(\cat{E}))$ is a t-structure. 
The \emph{$\ext{\cat{E}}$-kernel} of a morphism $M \to N$ in an exact category $\cat{E}$ is a morphism $L \to M$ such that $L \to M \to N$ is left $\ext{\cat{E}}$-acyclic. As the heart of a t-structure is abelian, the following result recovers \cite[Proposition~6]{Rump:2021}.

\begin{proposition} \label{t_str_iff_Ext_kernel}
Let $\cat{E}$ be an exact category. Then for any $* \in \{\emptyset,-\}$, the pair $(\cat{U}^*(\cat{E}),\cat{V}^*_\ell(\cat{E}))$ is a t-structure on $\dcat[*]{\cat{E}}$ if and only if every morphism in $\cat{E}$ has an $\ext{\cat{E}}$-kernel. 
\end{proposition}
\begin{proof}
We assume every morphism in $\cat{E}$ has an $\ext{\cat{E}}$-kernel. Let $X \in \dcat[*]{\cat{E}}$. We define a complex $V \in \cat{V}^*_\ell(\cat{E})$ as follows: Let $V^n = X^n$ and $d_V^n = d_X^n$ for $n \geq -1$. For $n < -1$ let $d_V^n$ be the $\ext{\cat{E}}$-kernel of $d_V^{n+1}$. By construction $V \in \cat{V}^*_\ell(\cat{E})$ and we have a commutative diagram
\begin{equation*}
\begin{tikzcd}
\cdots \ar[r] & X^{-3} \ar[r,"d_X^{-3}"] & X^{-2} \ar[r,"d_X^{-2}"] & X^{-1} \ar[r,"d_X^{-1}"] \ar[d,"="] & X^0 \ar[r,"d_X^0"] \ar[d,"="] & X^1 \ar[r] \ar[d,"="] & \cdots \\
\cdots \ar[r] & V^{-3} \ar[r,"d_V^{-3}"] & V^{-2} \ar[r,"d_V^{-2}"] & V^{-1} \ar[r,"d_V^{-1}"] & V^0 \ar[r,"d_V^0"] & V^1 \ar[r] & \cdots \nospacepunct{.}
\end{tikzcd}
\end{equation*}
Applying \cref{induced_mor_extresn} to this diagram, there exists an $\cat{E}$-quasi-isomorphism $g \colon W \to X$ and a chain map $f \colon W \to V$ extending this diagram; this means $f^n = g^n = \id$ for $n \geq -1$. By construction, the mapping cone of $f$ is $\cat{E}$-acyclic in non-positive degrees, hence $\cone(f) \in \susp^2 \cat{U}^*(\cat{E})$. It remains to observe that $X \xleftarrow{\sim} W \to V$ is a morphism in $\dcat[*]{\cat{E}}$. Hence $\susp^{-1} \cone(f) \to X \to V \to \cone(f)$ is the desired exact triangle in $\dcat[*]{\cat{E}}$.

Now assume that $(\cat{U}(\cat{E}),\cat{V}_\ell(\cat{E}))$ is a t-structure. Let $f \colon L \to M$ be a morphism in $\cat{E}$. We consider the complex $X$ with $f$ the differential in degree $-1$ and zero otherwise. Then there exists an exact triangle $U \to X \to V \to \susp U$ with $U \in \susp \cat{U}(\cat{E})$ and $V \in \cat{V}_\ell(\cat{E})$. Without loss of generality we may assume $U^n = 0$ for $n \geq 0$. Then
\begin{equation*}
\cone(U \to X) = (\cdots \to U^{-2} \xrightarrow{-d_U^{-2}} U^{-1} \to L \xrightarrow{f} M \to 0 \to \cdots) \cong V \in \cat{V}_\ell(\cat{E})\,.
\end{equation*}
Hence $U^{-1} \to L$ is an $\ext{\cat{E}}$-kernel of $f$. 
\end{proof}

\subsection{The heart} \label{sec:heart}

Let $\cat{E}$ be an exact category. By \cite{Dyer:extrinotes}, the triangulated structure of $\dcat[*]{\cat{E}}$ induces an exact structure on $\cat{U}^*(\cat{E}) \cap \cat{V}^*_\ell(\cat{E})$. Explicitly, a sequence $L \to M \to N$ is a conflation if and only if it fits into an exact triangle $L \to M \to N \to \susp L$ in $\dcat[*]{\cat{E}}$. We call this exact structure the \emph{admissible exact structure (in $\dcat[*]{\cat{E}}$)}. 

We observe
\begin{equation*}
\cat{U}(\cat{E}) \cap \cat{V}_\ell(\cat{E}) = \cat{U}^-(\cat{E}) \cap \cat{V}^-_\ell(\cat{E}) \quad \text{and} \quad \cat{U}^+(\cat{E}) \cap \cat{V}^+_\ell(\cat{E}) = \cat{U}^b(\cat{E}) \cap \cat{V}^b_\ell(\cat{E})
\end{equation*}
as the boundedness to the right has no effect on the intersection.

\begin{definition}
Let $\cat{E}$ be an exact category. We denote by
\begin{enumerate}
\item $\lheart{\cat{E}}$ the exact category whose underlying additive category is $\cat{U}(\cat{E}) \cap \cat{V}_\ell(\cat{E})$ and that is equipped with the admissible exact structure; this means $\lheart{\cat{E}}$ contains the complexes of $\dcat{\cat{E}}$ that are left $\ext{\cat{E}}$-acyclic in negative degrees, and $\cat{E}$-acyclic in positive degrees;
\item $\lheart[b]{\cat{E}}$ the exact category whose underlying additive category is $\cat{U}^b(\cat{E}) \cap \cat{V}^b_\ell(\cat{E})$ and that is equipped with the admissible exact structure; this means $\lheart[b]{\cat{E}}$ contains the bounded complexes which are $\cat{E}$-acyclic in positive degrees and degrees $< n$ for some integer $n \leq 0$, and left $\ext{\cat{E}}$-acyclic in negative degrees.
\end{enumerate}
We denote by $\rheart{\cat{E}}$ and $\rheart[b]{\cat{E}}$ the dual constructions.
\end{definition}

We obtain a sequence of fully exact subcategories
\begin{equation*}
\cat{E} \subseteq \lheart[b]{\cat{E}} \subseteq \lheart{\cat{E}}\,.
\end{equation*}

\begin{lemma} \label{heart_resolving}
Let $\cat{E}$ be an exact category. Then 
\begin{enumerate}
\item \label{heart_resolving:all} $\cat{E} \subseteq \lheart{\cat{E}}$ is a resolving subcategory; 
\item \label{heart_resolving:infty} $\cat{E} \subseteq \lheart[b]{\cat{E}}$ is a finitely resolving subcategory. 
\end{enumerate}
\end{lemma}
\begin{proof}
We first show that $\cat{E}$ is resolving in $\lheart{\cat{E}}$. Let $X \in \lheart{\cat{E}}$. We may assume that $X$ is concentrated in non-positive degrees. Brutal truncation yields the exact triangle
\[
\susp^{-1} X^{\leqslant -1} \to X^0 \to X \to X^{\leqslant -1}\,.
\]
As $\susp^{-1} X^{\leqslant -1} \in \lheart{\cat{E}}$ and $X^0 \in \cat{E} \subseteq \lheart{\cat{E}}$, this induces an $\lheart{\cat{E}}$-conflation. Therefore $X^0 \to X$ is an $\lheart{\cat{E}}$-deflation and \cref{resolving:generator} is satisfied.

Let $X \xrightarrow{i} M \xrightarrow{p} N$ be an $\lheart{\cat{E}}$-conflation with $M,N \in \cat{E}$. This means there is an exact triangle $X \xrightarrow{i} M \xrightarrow{p} N \to \susp X$ in $\dcat{\cat{E}}$. Hence 
\begin{equation*}
X \cong \susp^{-1} \cone(p) = (\cdots \to 0 \to M \xrightarrow{-p} N \to 0 \to \cdots) \in \lheart{\cat{E}}\,,
\end{equation*}
where $M$ and $N$ lie in degrees 0 and 1, respectively. As the complex is $\cat{E}$-acyclic in degree 1, the morphism $p$ is an $\cat{E}$-deflation. In particular, $p$ has a kernel in $\cat{E}$ and $X \cong \ker(p) \in \cat{E}$. Hence \cref{resolving:deflationclsd} holds. This finishes the proof for \cref{heart_resolving:all}. 

As $\lheart[b]{\cat{E}}$ is a fully exact subcategory of $\lheart{\cat{E}}$, $\cat{E}$ is a resolving subcategory of $\lheart[b]{\cat{E}}$. By definition every object in $\lheart[b]{\cat{E}}$ is isomorphic to a complex
\begin{equation*}
X = (\cdots \to 0 \to X^{n} \to X^{n+1} \to \cdots X^{-1} \to X^0 \to 0 \to \cdots)
\end{equation*}
for some $n \leq 0$ that is left $\ext{\cat{E}}$-acyclic in negative degrees. As $X^{\leqslant i-1} \to \susp^{-i} X^{i} \to X^{\leqslant i}$ is an $\lheart[b]{\cat{E}}$-conflation for every $0 \geq i \geq n$, the sequence
\begin{equation*}
\cdots \to 0 \to X^{n} \to X^{n+1} \to \cdots X^{-1} \to X^0 \to X
\end{equation*}
is $\lheart[b]{\cat{E}}$-acyclic with $X^{i} \in \cat{E}$ for any $0 \geq i \geq n$. So $\cat{E} \subseteq \lheart[b]{\cat{E}}$ is finitely resolving.
\end{proof}

\begin{proposition} \label{rcat_same_heart}
Let $\cat{F}$ be an exact category. 
\begin{enumerate}
\item If $\cat{E} \subseteq \cat{F}$ is a resolving subcategory, then there is an equivalence $\lheart{\cat{E}} \to \lheart{\cat{F}}$ of exact categories.
\item If $\cat{E} \subseteq \cat{F}$ is a finitely resolving subcategory, then there is an equivalence $\lheart[b]{\cat{E}} \to \lheart[b]{\cat{F}}$ of exact categories. 
\end{enumerate}
\end{proposition}
\begin{proof}
Let $\cat{E} \subseteq \cat{F}$ by resolving. By \cite[Theorem~3.11]{Henrard/vanRoosmalen:2020c}, the inclusion $\cat{E} \subseteq \cat{F}$ lifts to an equivalence $\dcat[-]{\cat{E}} \to \dcat[-]{\cat{F}}$ of triangulated categories. By \cref{resolving_ext_acyc}, this equivalence restricts to a functor $\lheart{\cat{E}} \to \lheart{\cat{F}}$. As the restriction is essentially surjective by \cref{eq:resolving_resn}, it is an equivalence. 

When $\cat{E} \subseteq \cat{F}$ is finitely resolving, then, by the same argument, the equivalence restricts to an equivalence $\lheart[b]{\cat{E}} \to \lheart[b]{\cat{F}}$. 
\end{proof}

From \cref{heart_resolving} we immediately obtain the following \namecref{heart_idempotent}.

\begin{corollary} \label{heart_idempotent}
Let $\cat{E}$ be an exact category. Then there are equivalences
\begin{equation*}
\lheart{\cat{E}} \to \lheart{\lheart{\cat{E}}} \quad \text{and} \quad \lheart[b]{\cat{E}} \to \lheart[b]{\lheart[b]{\cat{E}}}
\end{equation*}
of exact categories. \qed
\end{corollary}

As $\cat{E}$ is resolving in $\lheart{\cat{E}}$ and finitely resolving in $\lheart[b]{\cat{E}}$ there are triangle equivalences
\begin{equation*}
\dcat[-]{\cat{E}} \to \dcat[-]{\lheart{\cat{E}}} \quad \text{and} \quad \dbcat{\cat{E}} \to \dbcat{\lheart[b]{\cat{E}}}
\end{equation*}
by \cite[Theorem~3.11]{Henrard/vanRoosmalen:2020c}. It is straightforward to check that a quasi-inverse is given by taking the total complex. 

\begin{lemma}
Let $\cat{E}$ be an exact category and let $X$ be a bounded complex over $\lheart[b]{\cat{E}}$. Then $X$ is $\lheart[b]{\cat{E}}$-acyclic in negative degrees if and only if the total complex of $X$ is left $\ext{\cat{E}}$-acyclic in negative degree. 
\end{lemma}
\begin{proof}
We use induction on $n \leq 0$ to show the claim for a complex $X$ with $X^{<n} = 0$. 

For $n = 0$, the claim holds by the construction of the total complex using the cone and \cref{ConeExtacyclic}. 

We assume the claim holds for some $n \leq 0$. Assume $X^{<n-1} = 0$. The complex $X$ is $\lheart[b]{\cat{E}}$-acyclic in degree $n$ if and only if the differential $d_X^{n-1}$ is an $\lheart[b]{\cat{E}}$-inflation if and only if $\cone(d_X^{n-1}) \in \lheart[b]{\cat{E}}$. Hence $X$ is $\lheart[b]{\cat{E}}$-acyclic in negative degrees if and only if
\begin{equation*}
\tilde{X} = (\cdots \to 0 \to \cone(d_X^{n-1}) \to X^{n+1} \to \cdots)
\end{equation*}
is a complex $\lheart[b]{\cat{E}}$ and is $\lheart[b]{\cat{E}}$-acyclic in negative degrees. The total complex of $\tilde{X}$ is isomorphic to the total complex of $X$. The claim now follows from induction.
\end{proof}

\begin{corollary} \label{coaisle_same_for_lheart}
Let $\cat{E}$ be an exact category. Then the equivalence $\dbcat{\cat{E}} \to \dbcat{\lheart[b]{\cat{E}}}$ restricts to an equivalence $\cat{V}^b_\ell(\cat{E}) \to \cat{V}^b(\lheart[b]{\cat{E}})$. \qed
\end{corollary}

\subsection{Maximally non-negative subcategories of triangulated categories} 

We conclude this section by considering some of the properties of $\lheart[*]{\cat{E}}$, the dual construction $\rheart[*]{\cat{E}}$ and the pairs in \cref{ext_acyc_t_pair}. 

\begin{definition}
Let $\cat{T}$ be a triangulated category. A full subcategory $\cat{C}$ in $\cat{T}$ is called \emph{non-negative} if $\Hom_\cat{T}(\cat{C}, \susp^{<0} \cat{C})=0$. We say $\cat{C}$ is \emph{maximally non-negative} if for every non-negative $\cat{D} \subseteq \cat{T}$ we have $\cat{C} \subseteq \cat{D}$ implies $\cat{C}=\cat{D}$. 
\end{definition}

This property has previously considered by \cite{Jorgensen:2021}.

By construction, the subcategories $\lheart[*]{\cat{E}}$ and $\rheart[*]{\cat{E}}$ are non-negative in $\dcat[*]{\cat{E}}$. However, they need not be maximally non-negative in general. We have the following characterization of maximally non-negative subcategories of $\dbcat{\cat{E}}$. 

\begin{theorem} \label{char_maxl_nonneg}
Let $\cat{E}$ be a weakly idempotent complete exact category. The following are equivalent:
\begin{enumerate}
\item \label{char_maxl_nonneg:nonneg} $\cat{E}$ is maximally non-negative in $\dbcat{\cat{E}}$; 
\item \label{char_maxl_nonneg:heart} $\cat{E} = \lheart[b]{\cat{E}} = \rheart[b]{\cat{E}}$; 
\item \label{char_maxl_nonneg:aisle} $\cat{V}_\ell^b(\cat{E}) = \cat{V}^b(\cat{E})$ and $\cat{U}_r^b(\cat{E}) = \cat{U}^b(\cat{E})$; 
\item \label{char_maxl_nonneg:epimono} every monomorphism in $\cat{E}$ is an  $\cat{E}$-inflation and every epimorphism in $\cat{E}$ is an $\cat{E}$-deflation.
\end{enumerate}
\end{theorem}

The equivalence of \cref{char_maxl_nonneg:heart,char_maxl_nonneg:aisle,char_maxl_nonneg:epimono} follows from:

\begin{lemma} \label{char_lheart_mono}
Let $\cat{E}$ be a weakly idempotent complete exact category. The following are equivalent:
\begin{enumerate}
\item \label{char_lheart_mono:heart} $\cat{E} = \lheart[b]{\cat{E}}$; 
\item \label{char_lheart_mono:coaisle} $\cat{V}_\ell^b(\cat{E}) = \cat{V}^b(\cat{E})$; 
\item \label{char_lheart_mono:mono} every monomorphism in $\cat{E}$ is an $\cat{E}$-inflation.
\end{enumerate}
\end{lemma}
\begin{proof}
The implication \cref{char_lheart_mono:coaisle} $\implies$ \cref{char_lheart_mono:heart} is clear. 

We assume \cref{char_lheart_mono:heart} holds. Let $f \colon L \to M$ be a monomorphism. We consider the two-term complex with $d^{-1} = f$. This complex is left $\ext{\cat{E}}$-acyclic complex is negative degrees by \cref{mono_ext_acyc_hom_acyc}. By assumption it is quasi-isomorphic to a complex concentrated in degree 0. Hence $f$ is an $\cat{E}$-inflation.

We assume \cref{char_lheart_mono:mono} holds. Let $X \in \cat{V}_\ell^b(\cat{E})$ and $n$ the minimal degree in which $X$ is not $\cat{E}$-acyclic. Then $X$ is quasi-isomorphic to a complex $Y$ with $Y_{<n} = 0$. We assume $n < 0$. Then $d_Y^n$ has to be a monomorphism, and by assumption an $\cat{E}$-inflation. This is a contradiction, and hence $n = 0$ and $Y \in \cat{V}^b(\cat{E})$.
\end{proof}

\begin{proof}[Proof of \cref{char_maxl_nonneg}]
We assume \cref{char_maxl_nonneg:nonneg} holds. Since $\lheart[b]{\cat{E}} \supseteq \cat{E}$ is non-negative, equality holds. Analogously we get $\rheart[b]{\cat{E}} = \cat{E}$. 

We assume \cref{char_maxl_nonneg:aisle} holds. Let $\cat{C} \supseteq \cat{E}$ be a non-negative subcategory of $\dbcat{\cat{E}}$. From \cref{ext_acyc_t_pair} we obtain
\begin{equation*}
\cat{C} \subseteq (\susp^{>0} \cat{C})^\perp \subseteq (\susp^{>0} \cat{E})^\perp = (\susp \cat{U}^b(\cat{E}))^\perp = \cat{V}^b(\cat{E})\,.
\end{equation*}
Using the dual of \cref{ext_acyc_t_pair} we obtain $\cat{C} \subseteq \cat{U}^b(\cat{E})$. Hence $\cat{C} \subseteq \cat{U}^b(\cat{E}) \cap \cat{V}^b(\cat{E}) = \cat{E}$. 
\end{proof}

\begin{remark}
It follows from \cref{char_maxl_nonneg}: If $\cat{E}$ is maximally non-negative in $\dbcat{\cat{E}}$, then the exact structure of $\cat{E}$ is the maximal exact structure on its underlying additive category. It is easy to see that in this case $\cat{E}$ is abelian if and only if it is pre-abelian. 

In fact, if $\cat{E}$ is the maximal exact structure on a quasi-abelian category which is not an abelian category, then it is not maximally non-negative in $\dbcat{\cat{E}}$.
\end{remark}

An important example of maximally non-negative subcategories are the hearts of t-structures. More generally, we obtain non-negative subcategories through a weaker notion than t-structures:

\begin{definition}
Let $\cat{T}$ be a triangulated category. We call a pair of two full subcategories $(\cat{U}, \cat{V})$ of $\cat{T}$ a \emph{t-pair} if
\begin{enumerate}
\item $\cat{U}$ and $\cat{V}$ are closed under direct summands, isomorphisms and extensions; 
\item $\susp \cat{U} \subseteq \cat{U}$ and $\cat{V} \subseteq \susp \cat{V}$; and
\item $\Hom_\cat{T}(\susp \cat{U},\cat{V})=0$.
\end{enumerate}
We call $\cat{H} = \cat{U}\cap \cat{V}$ the \emph{heart} of the t-pair $(\cat{U},\cat{V})$.

We say a t-pair $(\cat{U},\cat{V})$ is \emph{left maximal} if  $(\susp \cat{U})^{\perp} = \cat{V}$. It is \emph{right maximal} if $\susp \cat{U} = {}^{\perp} \cat{V}$. We say it is \emph{maximal} if it is left and right maximal.   
\end{definition}

In \cref{ext_acyc_t_pair} we have shown that $(\cat{U}^*(\cat{E}),\cat{V}^*_\ell(\cat{E}))$ is a left maximal t-pair in $\dcat[*]{\cat{E}}$, and dually $(\cat{U}^*_r(\cat{E}),\cat{V}^*(\cat{E}))$ is a right maximal t-pair in $\dcat[*]{\cat{E}}$. By applying these constructions after each other, we get a maximal t-pair:

\begin{theorem} \label{maxl_tpair}
Let $\cat{E}$ be a weakly idempotent complete exact category. Then $(\cat{U}^b_r(\lheart[b]{\cat{E}}),\cat{V}^b_\ell(\cat{E}))$ and $(\cat{U}^b_r(\cat{E}),\cat{V}^b_\ell(\rheart[b]{\cat{E}}))$ are maximal t-pairs in $\dbcat{\cat{E}}$.
\end{theorem}
\begin{proof}
From \cref{ext_acyc_t_pair} applied to $(\lheart[b]{\cat{E}})^\op$, we get (1), (2) and the second part of (3), as well as $(\susp \cat{U}^b_r(\lheart[b]{\cat{E}}))^\perp \supseteq \cat{V}_\ell(\cat{E})$ using \cref{coaisle_same_for_lheart}. Further, we get
\begin{equation*}
\susp \cat{U}^b_r(\lheart[b]{\cat{E}}) \supseteq \susp \cat{U}^b(\cat{E}) \quad \implies \quad (\susp \cat{U}^b_r(\lheart[b]{\cat{E}}))^\perp \subseteq (\susp \cat{U}^b(\cat{E}))^\perp = \cat{V}^b_\ell(\cat{E})\,;
\end{equation*}
the latter again holds by \cref{ext_acyc_t_pair}.
\end{proof}

The hearts of the t-pairs in \cref{maxl_tpair} are $\rheart[b]{\lheart[b]{\cat{E}}}$ and $\lheart[b]{\rheart[b]{\cat{E}}}$. These do not need to coincide in general:

\begin{example} \label{ex_LHRH_not_RHLH}
Let $k$ be a field and $\Lambda = k (1 \leftarrow 2)$ be the finite-dimensional $k$-algebra given by quiver representations of Dynkin type $A_2$. The category of finite-dimensional $\Lambda$-modules $\mod{\Lambda}$ is Krull--Schmidt and has three indecomposable objects $P(1)$, $P(2)=I(1)$ and $I(2)$. 
The bounded derived category $\dbcat{\mod{\Lambda}}$ is an Auslander--Reiten category; this means it is Krull--Schmidt and every indecomposable object fits into an almost split triangle. The Auslander--Reiten quiver is
\[
\begin{tikzcd}[column sep=small]
\susp^{-1} I(2) \arrow[rd] \arrow[rr, dashed] & & P(2) \arrow[rd] \arrow[rr, dashed] & & \susp P(1) \arrow[rd] \arrow[rr, dashed] & & \susp I(2) \\
\cdots & P(1) \arrow[ru] \arrow[rr, dashed] & & I(2) \arrow[ru] \arrow[rr,dashed] & & \susp P(2) \arrow[ru] & \cdots                            
\end{tikzcd}
\]
where the dashed arrows point from the first to the third object of an exact triangle. 

We consider the exact category $\cat{E}=\add (I(1) \oplus  I(2))$. Then 
\begin{equation*}
\begin{aligned}
\rheart[b]{\cat{E}} &= \add(P(1)\oplus P(2)\oplus I(2)) \cong \lmod{\Lambda} \quad \text{and} \\
\lheart[b]{\cat{E}} &= \add(P(2)\oplus I(2)\oplus  \susp P(1))\,.
\end{aligned}
\end{equation*}
Both are hearts of a t-structure and therefore maximally non-negative. In particular, we have 
\[
\lheart[b]{\rheart[b]{\cat{E}}} = \rheart[b]{\cat{E}} \neq \lheart[b]{\cat{E}} = \rheart[b]{\lheart[b]{\cat{E}}}\,.
\]
\end{example}

The bounded derived equivalence between $\rheart[b]{\lheart[b]{\cat{E}}}$ and $\lheart[b]{\rheart[b]{\cat{E}}}$ can be seen as a generalization of derived equivalences for Artin algebras induced from tilting modules of projective dimension at most $n$ and, more generally, of those induced by tilting objects in abelian categories, cf. the following example. 

\begin{example}
Let $A$ be a finite-dimensional algebra over a field $k$ and $\lmod{A}$ the category of finite-dimensional left $A$-modules. For $n \geq 1$ let $T \in \lmod{A}$ be an $n$-tilting module, that is $T$ is self-orthogonal ($\Ext^i_A(T,T)=0$ for all $i \geq 1$), $\pdim_A T\leq n$ and there exists an exact sequence $0 \to A \to T^0 \to \cdots \to T^n \to 0 $ with $T^i \in \add (T)$, cf.\@ \cite[Chapter III]{Happel:1988}. The category 
\begin{equation*}
T^{\perp_{>0}}
\coloneqq \set{M \in \lmod{A}}{\Ext^i_{A}(T, X)=0 \, \forall i\geq 1} \subseteq \lmod{A}
\end{equation*}
is extension-closed and we let $\cat{E}$ be $T^{\perp_{>0}}$ equipped with the fully exact structure. Then $\cat{E}$ is coresolving in $\lmod{A}$. Moreover, for any $M\in \lmod{A}$ taking an injective resolution yields an exact sequence
\[
0 \to M \to I^0 \to I^1 \to \cdots \to I^{n-1} \to \Omega^{-n}M \to 0
\]
where $I^i$ are injective $A$-modules for $0 \leq i \leq n-1$. Moreover, by dimension shift $\Ext^i_A (T, \Omega^{-n} M)=\Ext^{i+n}_A (T,M)=0$ for all $i \geq 1$ as $\pdim T\leq n$. So  $\Omega^{-n} M \in \cat{E}$ and $\cat{E} \subseteq \lmod{A}$ is \emph{$n$-coresolving}; cf.\@ \cref{n_res_completion}. In particular,
\[
\lmod{A} = \rheart[b]{\cat{E}}= \lheart[b]{\rheart[b]{\cat{E}}} .
\]

Let $B \coloneqq \End_A(T)^\op$. Then $T$, viewed as a left $B$-module, is $n$-tilting by \cite[Theorem~1.5]{Miyashita:1986}. Equivalently, $C \coloneqq \Hom_k(T,k) \in \lmod{B}$ with ${\rm D}=\Hom_k(-,k)$ is an $n$-cotilting module; see for example \cite[Section~2]{AngeleriHuegel/Coelho:2001}. 
Dually to before, 
\begin{equation*}
{}^{{}_{0<}\perp} C \coloneqq 
\set{N \in \lmod{B}}{\Ext^i_{B}(N, C)=0 \, \forall i\geq 1} \subseteq \lmod{B}
\end{equation*}
is extension-closed and we let $\cat{F}$ be ${}^{{}_{0<}\perp} C$ with the fully exact structure. Then $\cat{F}$ is $n$-resolving in $\lmod{B}$. In particular,
\begin{equation*}
\lmod{B} = \lheart[b]{\cat{F}}= \rheart[b]{\lheart[b]{\cat{F}}}\,.
\end{equation*}

A classical result of Miyashita \cite[Theorem~1.16]{Miyashita:1986} gives an equivalence of exact categories 
\[
\Hom_A(T,-) \colon \cat{E} \to \cat{F}\,.
\]
Therefore, the induced derived equivalence between $\lmod{A}$ and $\lmod{B}$ is implied by our construction: 
\begin{equation*}
A\text{-}{\rm mod} = \rheart[b]{\cat{E}} \supseteq \cat{E} \cong \cat{F} \subseteq \lheart[b]{\cat{F}} = B\text{-}{\rm mod}
\end{equation*}
induces triangle equivalences
\begin{equation*}
\dcat[*]{A\text{-}{\rm mod}} \leftarrow \dcat[*]{\cat{E}} \cong \dcat[*]{\cat{F}} \to \dcat[*]{B\text{-}{\rm mod}}
\end{equation*}
for $* \in \{\emptyset,b\}$ using \cite{Henrard/vanRoosmalen:2020a}

Finitely generated modules over \emph{artinian} rings are dualizing $R$-varieties in the sense of Auslander; see \cite{Auslander/Reiten:1974}. Using this we found the $n$-cotilting module $C$. In general, this is not possible.

An object $T$ in an abelian category $\cat{A}$ is $n$-tilting if $\cat{E}=T^{\perp_{>0}}$ has enough projectives given by $\add (T)$ and $\cat{E}$ is $n$-coresolving. 
Assume $B= \End_{\cat{A}}(T)^\op$ is noetherian, so the category of finitely generated left $B$-modules is abelian. Then $\Hom_{\cat{A}} (T, -) \colon \cat{E} \to \lmod{B} $ is fully faithful, exact and identifies $\cat{E}$ with a resolving subcategory $\cat{F}$ of $\lmod{B}$. This yields
\begin{equation*}
\cat{A} = \rheart[b]{\cat{E}} \supseteq \cat{E} \cong \cat{F} \subseteq \lheart{\cat{F}} = \lmod{B}\,.
\end{equation*}
This need not yield derived equivalences, as the boundedness conditions do not match. If $\cat{F}$ is finitely resolving, then $\lheart{\cat{F}} = \lheart[b]{\cat{F}}$ and we get a bounded derived equivalence between $ \cat{A}$ and $\lmod{B}$ by \cite{Henrard/vanRoosmalen:2020a}. For example, if $\operatorname{gldim}(B) < \infty$ then $\cat{F}$ is finitely resolving. Instances of this are given by tilting bundles in coherent sheaves on a noetherian scheme, cf.\@ \cite{Beilinson:1978,Baer:1988}. 
\end{example}

Next we provide an example of maximally non-negative subcategories of $\dbcat{\cat{E}}$ that are not abelian. 

\begin{example} \label{ex_maxl_nonneg_not_abelian}
Let $\Lambda = k[T]/(T^2)$ for any field $k$ and $\cat{E} = \add(\Lambda) \subseteq \lmod{\Lambda}$. The only exact structure on $\cat{E}$ is the split exact structure and all epimorphisms and all monomorphisms are split. Hence, by \cref{char_maxl_nonneg}, the category $\cat{E}$ is maximally non-negative in $\dbcat{\cat{E}}$. Yet, the category $\cat{E}$ is not abelian because there is a morphism, namely $\Lambda \xrightarrow{\cdot T} \Lambda$, which has no epimorphism-monomorphism factorization.
On the other hand $\cat{E}$ is not maximally non-negative in $\dcat[-]{\cat{E}} = \dcat[-]{\lmod{\Lambda}}$ as it is properly contained in $\lmod{\Lambda}$.  
\end{example}

This \namecref{ex_maxl_nonneg_not_abelian} shows the t-pairs in \cref{maxl_tpair} need not be t-structures in general. Moreover, since the t-pairs are maximal, they are not contained in a t-structure. 

\begin{question}
Let $\cat{E}$ be an exact category, that is not abelian and maximally non-negative in $\dbcat{\cat{E}}$. Does there exists a bounded t-structure on $\dbcat{\cat{E}}$? Or equivalently, does there exists an abelian maximally non-negative subcategory in $\dbcat{\cat{E}}$?
\end{question}

\section{Functorial view on Ext-resolutions}

Let $\cat{A}$ be a small additive category. The category of right $\cat{A}$-modules is the category of additive functors $\Mod{\cat{A}} \colonequals \Add(\cat{A}^\op,\abcat)$. By \cite[Theorem~5.11]{Freyd:2003}, the category $\Mod{\cat{A}}$ is abelian and the (co)kernels are precisely the pointwise (co)kernels. We denote by
\begin{equation*}
\yoneda{-} \colon \cat{A} \to \Mod{\cat{A}} \,,\quad M \mapsto \Hom_{\cat{A}}(-,M)
\end{equation*}
the \emph{Yoneda embedding}. The functors $\Hom_{\cat{A}}(-,X)$ are called \emph{representable}. In the functor category $\Mod{\cat{A}}$ the representable functors are projective. 

An object $\sF \in \Mod{\cat{A}}$ is called \emph{finitely presented}, if there exists a morphism $f \colon M \to N$ in $\cat{A}$, such that
\begin{equation*}
\yoneda{M} \xrightarrow{\yoneda{f}} \yoneda{N} \to \sF \to 0
\end{equation*}
is exact in $\Mod{\cat{A}}$. The subcategory of finitely presented objects in $\Mod{\cat{A}}$ is extension- and inflation-closed. We denote by $\mod{\cat{A}}$ the exact category of the finitely presented objects in $\Mod{\cat{A}}$ equipped with the fully exact structure. 

\subsection{Functors with an Ext-resolution} \label{sec:functors_ext_resn}

Let $\cat{E}$ be a small exact category. We are interested in subcategories of $\Mod{\cat{E}}$ consisting of functors that have an Ext-resolution; 
that is functors $\coker(\yoneda{d_X^{-1}})$ where $X$ is an $\ext{\cat{E}}$-resolution. We say such a functor \emph{admits an $\ext{\cat{E}}$-resolution}. For such functors, we prove an analogue of the horseshoe lemma. Unlike the classical case, the construction involves a quasi-isomorphism similar to the results in \cref{sec:mor_ext_rsn}.

\begin{lemma} \label{horseshoe}
Let $\cat{E}$ be a small exact category. Let $0 \to \sE \to \sF \to \sG \to 0$ be a short exact sequence in $\Mod{\cat{E}}$ such that $\sE$ and $\sG$ admit $\ext{\cat{E}}$-resolutions $X$ and $Z$, respectively. Then there exists a complex $W$, an $\cat{E}$-quasi-isomorphism $g \colon W \to Z$ and a chain map $f \colon \susp^{-1} W \to X$ such that $\sF \cong \coker(\yoneda{d_{\cone(f)}^{-1}})$. 

In particular, the complex $\cone(f)$ is an $\ext{\cat{E}}$-resolution of $\sF$.
\end{lemma}

\begin{proof}
We start with the same argument as for the classical horseshoe lemma; see for example \cite[Lemma~2.2.8]{Weibel:1994}. We obtain the commutative diagram
\begin{equation*}
\begin{tikzcd}[ampersand replacement=\&,row sep=small]
0 \ar[r] \&[-1em] \ker(\alpha') \ar[r,mono] \ar[d,mono] \& \ker(\beta') \ar[r,epi] \ar[d,mono] \& \ker(\gamma') \ar[d,mono] \ar[r] \&[-1em] 0 \\
0 \ar[r] \& \yoneda{X^{-1}} \ar[d,epi,"\alpha'"] \ar[r,mono,"{\begin{psmallmatrix} 1 \\ 0 \end{psmallmatrix}}"] \& \yoneda{X^{-1}} \oplus \yoneda{Z^{-1}} \ar[r,epi,"{\begin{psmallmatrix} 0 & 1 \end{psmallmatrix}}"] \ar[d,epi,"\beta'"] \& \yoneda{Z^{-1}} \ar[d,epi,"\gamma'"] \ar[r] \& 0 \\
0 \ar[r] \& \ker(\alpha) \ar[d,mono] \ar[r,mono] \& \ker(\beta) \ar[d,mono] \ar[r,epi] \& \ker(\gamma) \ar[d,mono] \ar[r] \& 0 \\
0 \ar[r] \& \yoneda{X^0} \ar[r,mono,,"{\begin{psmallmatrix} 1 \\ 0 \end{psmallmatrix}}"] \ar[d,epi,"\alpha"] \& \yoneda{X^0} \oplus \yoneda{Z^0} \ar[r,epi,"{\begin{psmallmatrix} 0 & 1 \end{psmallmatrix}}"] \ar[d,dashed,"\beta"] \& \yoneda{Z^0} \ar[d,epi,"\gamma"] \ar[r] \& 0 \\
0 \ar[r] \& \sE \ar[r,mono] \& \sF \ar[r,epi] \& \sG \ar[r] \& 0 \nospacepunct{.}
\end{tikzcd}
\end{equation*}
As the bottom two rows are exact by construction, the morphism $\beta$ exists and it is an epimorphism by the 5-lemma. The middle row is exact by the snake lemma, and the vertical maps to the middle row are epimorphisms by construction. Again by the snake lemma, the top row is exact. By the commutativity, the composition
\begin{equation*}
\yoneda{X^{-1}} \oplus \yoneda{Z^{-1}} \to \ker(\beta) \to \yoneda{X^0} \oplus \yoneda{Z^0}
\end{equation*}
is of the form $\yoneda{\begin{psmallmatrix} d_X^{-1} & f^0 \\ 0 & d_Z^{-1} \end{psmallmatrix}}$. 

The induced morphism $\yoneda{Z^{-2}} \to \ker(\gamma')$ factors through $\ker(\beta')$ as $\yoneda{Z^{-2}}$ is projective. Hence the composition $\yoneda{Z^{-2}} \to \yoneda{X^{-1}} \oplus \yoneda{Z^{-1}}$ yields a morphism $f^{-1} \colon Z^{-2} \to X^{-1}$ and we obtain the following commuting diagram
\begin{equation*}
\begin{tikzcd}
\susp^{-1} Z &[-3em] : &[-3em] \cdots \ar[r] & Z^{-3} \ar[r,"-d_Z^{-3}"] & Z^{-2} \ar[r,"{-d_Z^{-2}}"] \ar[d,"{f^{-1}}"] & Z^{-1} \ar[r,"-d_Z^{-1}"] \ar[d,"{f^0}"] & Z^0 \ar[r] \ar[d] & 0 \ar[r] \ar[d] & \cdots \\
X & : & \cdots \ar[r] & X^{-2} \ar[r,"d_X^{-2}"] & X^{-1} \ar[r,"{d_X^{-1}}"] & X^0 \ar[r] & 0 \ar[r] & 0 \ar[r] & \cdots \nospacepunct{.}
\end{tikzcd}
\end{equation*}
By \cref{induced_mor_extresn}, there exists an $\cat{E}$-quasi-isomorphism $g \colon \susp^{-1} W \to \susp^{-1} Z$ and a chain map $h \colon \susp^{-1} W \to X$ extending this diagram. This means $W^n = Z^n$ and $g^n = \id$ for $n \geq -2$. Hence $W$, $g$ and $f$ satisfy the claim.
\end{proof}

As a consequence of this \namecref{horseshoe}, the subcategories of $\Mod{\cat{E}}$ of functors with a (bounded) $\ext{\cat{E}}$-resolution are extension-closed. Hence we can define the following exact categories:

\begin{definition}
Let $\cat{E}$ be a small exact category. We denote by
\begin{enumerate}
\item $\Pmod{\cat{E}}$ the strictly full subcategory of $\Mod{\cat{E}}$ of functors which admit an $\ext{\cat{E}}$-resolution, equipped with the fully exact structure.
\item $\Pmod[b]{\cat{E}}$ the strictly full subcategory of $\Mod{\cat{E}}$ of functors which admit a bounded $\ext{\cat{E}}$-resolution, equipped with the fully exact structure.
\end{enumerate}
\end{definition}

There is a sequence of fully exact subcategories
\begin{equation*}
\Pmod[b]{\cat{E}} \subseteq \Pmod{\cat{E}} \subseteq \mod{\cat{E}} \subseteq \Mod{\cat{E}}\,.
\end{equation*}

\begin{lemma} \label{extending_mor_to_ext_resn}
Let $\cat{E}$ be a small exact category and $\sF \in \Pmod{\cat{E}}$. If $\sF \cong \coker(\yoneda{f})$ for some morphism $f \colon M \to N$ in $\cat{E}$, then there exists a split epimorphism $t$ and an $\ext{\cat{E}}$-resolution $X$ such that $d_X^{-1} = f \oplus t$. In particular, $X$ is an $\ext{\cat{E}}$-resolution of $\sF$.

If $\sF \in \Pmod[b]{\cat{E}}$, then the $\ext{\cat{E}}$-resolution can be chosen to be bounded.
\end{lemma}
\begin{proof}
As $\sF \in \Pmod{\cat{E}}$ there exists an $\ext{\cat{E}}$-resolution $Y$ with $\sF \cong \coker(\yoneda{d_Y^{-1}})$. Then there exist commutative squares
\begin{equation*}
\begin{tikzcd}
Y^{-1} \ar[r,"u^{-1}"] \ar[d,"d_Y^{-1}"] & M \ar[r,"v^{-1}"] \ar[d,"f"] & Y^{-1} \ar[d,"d_Y^{-1}"] \\
Y^0 \ar[r,"u^{0}"] & N \ar[r,"v^{0}"] & Y^0
\end{tikzcd}
\end{equation*}
and morphisms $\sigma \colon Y^0 \to Y^{-1}$ and $\tau \colon N \to M$ such that 
\begin{equation*}
\id_{Y^0} - v^0 u^0 = d_Y^{-1} \sigma \quad \text{and} \quad \id_N - u^0 v^0 = f \tau\,.
\end{equation*}
We obtain an isomorphism of complexes
\begin{equation*}
\begin{tikzcd}[ampersand replacement=\&,column sep=large]
Y^{-3} \ar[r,"{\begin{psmallmatrix} 0 \\ 0 \\ d_Y^{-3} \end{psmallmatrix}}"] \ar[d,"="] \& Y^0 \oplus M \oplus Y^{-2} \ar[r,"{\begin{psmallmatrix} 1 & 0 & 0 \\ 0 & 1 & 0 \\ 0 & 0 & d_Y^{-2} \\ 0 & 0 & 0 \end{psmallmatrix}}"] \ar[d,"="] \& Y^0 \oplus M \oplus Y^{-1} \oplus N \ar[r,"{\begin{psmallmatrix} 0 & 0 & d_Y^{-1} & 0 \\ 0 & 0 & 0 & 1 \end{psmallmatrix}}"] \ar[d,"\cong"] \& Y^0 \oplus N \ar[d,"\cong"] \\
Y^{-3} \ar[r] \& Y^0 \oplus M \oplus Y^{-2} \ar[r] \& Y^0 \oplus M \oplus Y^{-1} \oplus N \ar[r,"{\begin{psmallmatrix} 1 & 0 & 0 & 0 \\ 0 & f & 0 & 0 \end{psmallmatrix}}"] \& Y^0 \oplus N \nospacepunct{.}
\end{tikzcd}
\end{equation*}
The top row is the direct sum of $Y$ with a split acyclic complex; hence it is left $\ext{\cat{E}}$-acyclic in negative degrees. Hence, the bottom row is an $\ext{\cat{E}}$-resolution, and the $(-1)$st differential is $f \oplus t$ for a split epimorphism $t$. 
\end{proof}

\begin{lemma} \label{Pcat_defl_clsd}
Let $\cat{E}$ be a small exact category. Then $\Pmod{\cat{E}}$ and $\Pmod[b]{\cat{E}}$ are deflation-closed in $\mod{\cat{E}}$.
\end{lemma}
\begin{proof}
Let $\sE \to \sF \to \sG$ be a $\mod{\cat{E}}$-conflation. We assume $\sF, \sG \in \Pmod{\cat{E}}$ and $\sG$ admits an $\ext{\cat{E}}$-resolution $Z$. As $\sE$ is finitely presented, there exists a morphism $d_X^{-1} \colon X^{-1} \to X^0$ such that $\sE \cong \coker(\yoneda{d_X^{-1}})$. By the same argument as in the proof of \cref{horseshoe}, we obtain from $d_X^{-1}$ and $Z$ a commutative square
\begin{equation*}
\begin{tikzcd}
Z^{-2} \ar[r,"{-d_Z^{-2}}"] \ar[d,"{f^{-1}}"] & Z^{-1} \ar[d,"{f^0}"] \\
X^{-1} \ar[r,"{d_X^{-1}}"] & X^0 \nospacepunct{,}
\end{tikzcd}
\end{equation*}
such that $\sF \cong \coker(\yoneda{\begin{psmallmatrix} d_X^{-1} & f^0 \\ 0 & d_Z^{-1} \end{psmallmatrix}})$. Using \cref{extending_mor_to_ext_resn}, there exists a split epimorphism $t$ and an $\ext{\cat{E}}$-resolution $Y$ such that
\begin{equation*}
d_Y^{-1} = \begin{psmallmatrix} t & 0 & 0 \\ 0 & d_X^{-1} & f^0 \\ 0 & 0 & d_Z^{-1} \end{psmallmatrix}\,.
\end{equation*}
By replacing $d_X^{-1}$ by $t \oplus d_X^{-1}$, so may assume $d_Y^{-1} = \begin{psmallmatrix} d_X^{-1} & f^0 \\ 0 & d_Z^{-1} \end{psmallmatrix}$. From \cref{induced_mor_extresn} we obtain a complex $W$ with $d_W^{-1} = d_Y^{-1}$, an $\cat{E}$-quasi-isomorphism $W \to X$ and a morphism of complexes
\begin{equation*}
\begin{tikzcd}[ampersand replacement=\&]
\cdots \ar[r] \& W^{-2} \ar[r,"d_Y^{-2}"] \ar[d] \& X^{-1} \oplus Z^{-1} \ar[r,"{\begin{psmallmatrix} d_X^{-1} & f^0 \\ 0 & d_Z^{-1} \end{psmallmatrix}}"] \ar[d,"{\begin{psmallmatrix} 0 & 1 \end{psmallmatrix}}"] \&[+2em] X^0 \oplus Z^0 \ar[r] \ar[d,"{\begin{psmallmatrix} 0 & 1 \end{psmallmatrix}}"] \& 0 \ar[r] \ar[d] \& \cdots \\
\cdots \ar[r] \& Z^{-2} \ar[r,"d_Z^{-2}"] \& Z^{-1} \ar[r] \& Z^0 \ar[r] \& 0 \ar[r] \& \cdots \nospacepunct{.}
\end{tikzcd}
\end{equation*}
The desuspensed cone of this morphism is left $\ext{\cat{E}}$-acyclic in negative degrees and, after removing $Z^0 \xrightarrow{1} Z^0$, it is zero in positive degree. It is straightforward to check that its $(-1)$st differential is a presentation for $\sE$. Hence $\sE \in \Pmod{\cat{E}}$. It remains to observe that if $W$ and $Z$ were bounded, then so is their mapping cone. Hence $\sF, \sG \in \Pmod[b]{\cat{E}}$ implies $\sE \in \Pmod[b]{\cat{E}}$.
\end{proof}

The Yoneda functor restricts to functors $\cat{E} \to \Pmod{\cat{E}}$ and $\cat{E} \to \Pmod[b]{\cat{E}}$. However, these functors need not be (fully) exact embeddings, as $\cat{E}$-deflations need not be mapped to epimorphisms in $\Mod{\cat{E}}$. To resolve this, we consider localizations of $\Pmod{\cat{E}}$ and $\Pmod[b]{\cat{E}}$. 

\subsection{Localization by deflation-percolating subcategories}

We recall the localization of exact categories as introduced in \cite{Henrard/vanRoosmalen:2020a}. This construction yields an (inflation/deflation-)exact structure on the localization. However, we will not consider the induced exact structure, as it is not compatible with the exact structure on the localization of $\Mod{\cat{A}}$; see \cref{difference_exact_on_Qlcat}. 

Let $\cat{E}$ be an exact category. A full subcategory $\cat{C} \subseteq \cat{E}$ is \emph{deflation-percolating}, if it is a Serre subcategory and every morphism $L \to M$ in $\cat{E}$ with $M \in \cat{C}$ is $\cat{E}$-admissible; that is, it factors as an $\cat{E}$-deflation followed by an $\cat{E}$-inflation.

\begin{remark}
This is the definition of deflation-percolating following \cite{Henrard/Kvamme/vanRoosmalen:2022}. In \cite[Definition~6.1]{Henrard/vanRoosmalen:2020a} this notion is called admissibly deflation-percolating, and deflation-percolating is a weaker notion; see \cite[Lemma~6.3]{Henrard/vanRoosmalen:2020a}.
\end{remark}

Let $\cat{C}$ be a deflation-percolating subcategory of an exact category $\cat{E}$. The class $S_\cat{C}$ of finite compositions of $\cat{E}$-deflations with kernel in $\cat{C}$ and $\cat{E}$-inflations with cokernel in $\cat{C}$ is a right multiplicative system, also called a right calculus of fractions; see \cite[Theorem 6.14]{Henrard/vanRoosmalen:2020a}. Then
\begin{equation*}
\cat{E}/\cat{C} \colonequals \cat{E}[S_\cat{C}^{-1}]
\end{equation*}
is the \emph{localization of $\cat{E}$ by $\cat{C}$}. 

\begin{lemma} \label{deflclsd_loc_embedding}
Let $\cat{E} \subseteq \cat{F}$ be a deflation-closed, fully exact subcategory of an exact category $\cat{F}$. Let $\cat{C} \subseteq \cat{E}$ be a subcategory that is deflation-percolating in $\cat{F}$. Then $\cat{C}$ is deflation-percolating in $\cat{E}$ and the canonical functor $\cat{E}/\cat{C} \to \cat{F}/\cat{C}$ is an embedding that is full.
\end{lemma}
\begin{proof}
We first show that $\cat{C}$ is deflation-percolating in $\cat{E}$. As $\cat{C}$ is Serre in $\cat{F}$ and $\cat{E} \subseteq \cat{F}$ is a fully exact subcategory, $\cat{C}$ is Serre in $\cat{E}$ as well. Let $f \colon L \to M$ be a morphism in $\cat{E}$ with $M \in \cat{C}$. By assumption, $f$ is $\cat{F}$-admissible; that is there exists a factorization
\[
\begin{tikzcd}
& L \ar[dr,defl=\cat{F},"p" {near start,swap}] \ar[rr,"f"] & & M \ar[dr,defl=\cat{F},"q" {near start,swap}] \\
K \ar[ur,infl=\cat{F},"i" {near end,swap}] & & I \ar[ur,infl=\cat{F},"j" {near end,swap}] & & C \nospacepunct{,}
\end{tikzcd}
\]
where $(i,p)$ and $(j,q)$ are $\cat{F}$-conflations. As $M \in \cat{C}$ and $\cat{C}$ is Serre, we also have $I,C \in \cat{C}$. As $L,I \in \cat{E}$ and $\cat{E} \subseteq \cat{F}$ is deflation-closed, we have $K \in \cat{E}$. In particular, $(i,p)$ and $(j,q)$ are $\cat{E}$-conflations.

For the second claim it is enough to show that $\cat{E} \subseteq \cat{F}$ is right cofinal with respect to the right multiplicative system $S_\cat{C}$ in $\cat{F}$ by \cite[Lemma~1.2.5]{Krause:2022}; $\cat{E} \subseteq \cat{F}$ is called \emph{right cofinal with respect to $S_\cat{C}$} if for every $s \colon M \to N$ in $S_\cat{C}$ with $N \in \cat{E}$, there exists a $t \colon L \to M$ with $L \in \cat{E}$ such that $st \colon L \to N$ in $S_\cat{C}$. 

Let $s \colon M \to N$ in $S_\cat{C}$ with $N \in \cat{E}$. By \cite[Section~6.3]{Henrard/vanRoosmalen:2020a}, the morphism $s$ is $\cat{F}$-admissible such that $C=\coker(s)$ and $K=\ker(s)$ are in $\cat{C}$. This yields two $\cat{F}$-conflations
\[
\begin{tikzcd}
K \ar[r,infl=\cat{F}] & M \ar[r,defl=\cat{F}] & I
\end{tikzcd}
\quad \text{and} \quad
\begin{tikzcd}
I \ar[r,infl=\cat{F}] & N \ar[r,defl=\cat{F}] & C
\end{tikzcd}
\,.
\]
As $\cat{E}$ is deflation-closed in $\cat{F}$ and $N,C \in \cat{E}$, we also have $I \in \cat{E}$. As $\cat{E}$ is extension-closed in $\cat{F}$ and $K, I \in \cat{E}$, it follows that $M \in \cat{E}$. For $t = \id_M$, the composition $st \colon M \to N$ lies in $S_\cat{C}$ with $M,N \in \cat{E}$. Hence $\cat{E} \subseteq \cat{F}$ is right cofinal in $S_\cat{C}$.
\end{proof}

\subsection{Effaceable functors}

Let $\cat{E}$ be a small exact category. An object $\sF \in \Mod{\cat{E}}$ is \emph{effaceable}, if there exists an $\cat{E}$-deflation $d$ such that $\sF \cong \coker(\yoneda{d})$. We denote the strictly full subcategory of $\Mod{\cat{E}}$ of effaceable objects by $\eff{\cat{E}}$. The effaceable functors are the finitely presented objects of a bigger class of functors, namely, the locally effaceable functors. An object $\sF \in \Mod{\cat{E}}$ is \emph{locally effaceable} if for any $A \in \cat{E}$ and $x \in \sF(A)$, there exists an $\cat{E}$-deflation $p \colon B \to A$ such that $\sF(p)(x) = 0$. We denote the strictly full subcategory of locally effaceable objects by $\Eff{\cat{E}}$. One has
\begin{equation} \label{eff_is_Eff_cap_mod}
\eff{\cat{E}} = \Eff{\cat{E}} \cap \mod{\cat{E}}\,;
\end{equation}
see \cite[Lemma~2.13]{Enomoto:2021} and also \cite[Proposition~3.20]{Henrard/Kvamme/vanRoosmalen:2022}.

\begin{remark}
The notation for (locally) effaceable functors is not consistent in the literature. In \cite[2.2]{Grothendieck:1957} and \cite[Appendix~A]{Keller:1990} the locally effaceable functors are called effaceable, while in \cite[Definition~3.19]{Henrard/Kvamme/vanRoosmalen:2022} the locally effaceable functors are called weakly effaceable. Our notation follows \cite[Section~2.3]{Krause:2022}.
\end{remark}

By \cite[Appendix~A]{Keller:1990}, the subcategory $\Eff{\cat{E}} \subseteq \Mod{\cat{E}}$ is Serre, and, by \cite[Proposition~3.23, Corollary~3.24]{Henrard/Kvamme/vanRoosmalen:2022}, the subcategory $\eff{\cat{E}} \subseteq \mod{\cat{E}}$ is deflation-percolating. Hence the localizations exist, and
\begin{equation*}
\mod{\cat{E}}/\eff{\cat{E}} \to \Mod{\cat{E}}/\Eff{\cat{E}}
\end{equation*}
is a full embedding into an abelian category. Moreover, $\mod{\cat{E}}/\eff{\cat{E}}$ viewed as a subcategory of $\Mod{\cat{E}}/\Eff{\cat{E}}$ is extension-closed; this follows from \cite[Theorem~3.4]{Nordskova:2024}.

We denote by $\Qlcat{\cat{E}}$ the exact category with the underlying additive category $\mod{\cat{E}}/\eff{\cat{E}}$ equipped with the fully exact structure. This is the \emph{left quotient category} introduced in \cite[Section~3]{Rump:2020}. By \cite[Theorem~4]{Rump:2020}, $\cat{E}$ is semi-resolving in $\Qlcat{\cat{E}}$.

\begin{remark}
By \cite[Appendix~A]{Keller:1990}, the localization $\Mod{\cat{E}}/\Eff{\cat{E}}$ is equivalent to the subcategory of left exact functors of $\Mod{\cat{E}}$. By \cite[Theorem~3.4]{Nordskova:2024}, the subcategory $\mod{\cat{E}}/\eff{\cat{E}}$ can be identified with the subcategory of compact objects. 
\end{remark}

\begin{remark} \label{difference_exact_on_Qlcat}
The fully exact structure on $\mod{\cat{E}}/\eff{\cat{E}}$ need not coincide with the exact structure induced by the localization. The exact structure on $\Qlcat{\cat{E}}$ is the maximal exact structure, it consists of all kernel-cokernel pairs in $\Qlcat{\cat{E}}$. However, there exist kernel-cokernel pairs in $\Qlcat{\cat{E}}$ which are not induced from a conflation in $\mod{\cat{E}}$. For example, for $\sF \in \Pmod{\cat{E}}$ with an $\ext{\cat{E}}$-resolution $X$ the sequence
\begin{equation*}
\coker(d_X^{-2}) \to \yoneda{X^{0}} \to \sF
\end{equation*}
is a kernel-cokernel pair in $\mod{\cat{E}}/\eff{\cat{E}}$ by \cref{generating_Rmod}. However, it is not a kernel-cokernel pair in $\mod{\cat{E}}$ as the morphism $\coker(d_X^{-2}) \to \ker(\yoneda{X^0} \to \sF)$ need not be an isomorphism in $\mod{\cat{E}}$.
\end{remark}

\begin{corollary} \label{eff_defl_percolating}
Let $\cat{E}$ be a small exact category. Then $\eff{\cat{E}}$ is a deflation-percolating subcategory of $\Pmod{\cat{E}}$ and of $\Pmod[b]{\cat{E}}$.
\end{corollary}
\begin{proof}
This follows from \cref{deflclsd_loc_embedding} as $\eff{\cat{E}} \subseteq \mod{\cat{E}}$ is deflation-percolating, and $\Pmod{\cat{E}}$ and $\Pmod[b]{\cat{E}}$ are deflation-closed, fully exact subcategories of $\mod{\cat{E}}$ by \cref{Pcat_defl_clsd}.
\end{proof}

\section{In the left quotient category} \label{sec:rclosure}

We have shown that $\eff{\cat{E}}$ is deflation-percolating in the deflation-closed fully exact subcategories
\begin{equation*}
\Pmod[b]{\cat{E}} \subseteq \Pmod{\cat{E}} \subseteq \mod{\cat{E}}\,.
\end{equation*}
By \cref{deflclsd_loc_embedding}, we obtain a chain of subcategories
\begin{equation*} \label{chain_subcat_loc}
\Pmod[b]{\cat{E}}/\eff{\cat{E}} \subseteq \Pmod{\cat{E}}/\eff{\cat{E}} \subseteq \Qlcat{\cat{E}}\,.
\end{equation*}
We want to equip each localization in this chain with an exact structure. 

First, we show that the Yoneda embedding sends left $\ext{\cat{E}}$-acyclic sequences to acyclic sequences in the abelian category $\Mod{\cat{E}}/\eff{\cat{E}}$. In particular, $\ext{\cat{E}}$-resolutions become resolutions.

\begin{lemma} \label{ext_acyc_in_loc_acyc}
Let $\cat{E}$ be a small exact category and $L \xrightarrow{f} M \xrightarrow{g} N$ a sequence in $\cat{E}$ with $gf = 0$. The sequence is left $\ext{\cat{E}}$-acyclic, if and only if $\yoneda{L} \to \yoneda{M} \to \yoneda{N}$ is acyclic in $\Mod{\cat{E}}/\Eff{\cat{E}}$, or equivalently $\Qlcat{\cat{E}}$-acyclic.
\end{lemma}
\begin{proof}
We consider the commutative diagram
\begin{equation*}
\begin{tikzcd}
\yoneda{L} \ar[r,"\yoneda{f}"] & \yoneda{M} \ar[dr,two heads,"q" swap] \ar[rrr,"\yoneda{g}"] \ar[drr,two heads] &[-2em] &[-1em] &[-2em] \yoneda{N} \nospacepunct{.} \\
& & \coker(\yoneda{f}) \ar[r,two heads,"\hat{q}" swap] & \image(\yoneda{g}) \ar[ur,tail,"i"]
\end{tikzcd}
\end{equation*}
Let $j \colon \sK \to \coker(\yoneda{f})$ be the kernel of $\hat{q}$. Then $\yoneda{L} \to \yoneda{M} \to \yoneda{N}$ is acyclic if and only if $\sK$ is locally effaceable. 

We assume $L \to M \to N$ is left $\ext{\cat{E}}$-acyclic. Let $A \in \cat{E}$ and $x \in \sK(A)$. As $q$ is an epimorphism, there exists an element $a \in \yoneda{M}(A) = \Hom_{\cat{E}}(A,M)$ such that $j(A)(x) = p(A)(a)$. Then
\begin{equation*}
g a = \yoneda{g}(A)(a) = (i\hat{q}q)(A)(a) = (i\hat{q}j)(A)(x) = 0\,.
\end{equation*}
Since the sequence is left $\ext{\cat{E}}$-acyclic in $M$, there exists an $\cat{E}$-deflation $p \colon B \to A$ and a morphism $h \colon B \to L$ such that $a p = f h$. From the naturality we obtain the following commutative diagram
\begin{equation} \label{ext_acyc_in_loc_acyc:diagram}
\begin{tikzcd}[column sep=large]
\sK(A) \ar[r,"j(A)"] \ar[d,"\sK(p)"] & \coker(\yoneda{g})(A) \ar[d,"{\coker(\yoneda{g}(p)}"] & \ar[l,two heads,"q(A)" swap] \yoneda{M}(A) \ar[d,"\yoneda{M}(p)"] \\
\sK(B) \ar[r,"j(B)" swap] & \coker(\yoneda{g})(B) & \yoneda{M}(B) \ar[l,two heads,"q(B)"] & \yoneda{L}(B) \ar[l,"{\yoneda{g}(B)}"] \nospacepunct{.}
\end{tikzcd}
\end{equation}
A diagram chase yields $(j(B) \sK(p))(x) = 0$. 
Hence $\sK(p)(x) = 0$ and $\sK$ is locally effaceable.

For the converse, we assume $\sK$ is locally effaceable. Let $a \colon A \to M$ be a morphism with $ga = 0$. This means $\yoneda{g}(A)(a) = 0$. Hence $(\hat{q} q)(A)(a) = 0$ and there exists $x \in \sK(A)$ such that $j(A)(x) = q(A)(a)$. As $\sK$ is locally effaceable, there exists an $\cat{E}$-deflation $p \colon B \to A$ such that $\sK(p)(x) = 0$. A diagram chase along \cref{ext_acyc_in_loc_acyc:diagram} yields
\begin{equation*}
q(B)(ap) = (q(B) \yoneda{M}(p))(a) = 0\,.
\end{equation*}
Hence there exists $h \in \yoneda{L}(B)$ such that $gh = \yoneda{g}(B)(h) = ap$.
\end{proof}

\begin{lemma} \label{generating_Rmod}
Let $\cat{E}$ be a small exact category. 
\begin{enumerate}
\item For every $\sF \in \Pmod{\cat{E}}$ there exists a $\Qlcat{\cat{E}}$-conflation $\sE \to \yoneda{M} \to \sF$ with $\sE \in \Pmod{\cat{E}}$.
\item For every $\sF \in \Pmod[b]{\cat{E}}$ there exists a $\Qlcat{\cat{E}}$-conflation $\sE \to \yoneda{M} \to \sF$ with $\sE \in \Pmod[b]{\cat{E}}$.
\end{enumerate}
\end{lemma}
\begin{proof}
Assume that $\sF \in \Pmod{\cat{E}}$ and let $X$ be an $\ext{\cat{E}}$-resolution of $\sF$. By \cref{ext_acyc_in_loc_acyc}
\begin{equation*}
\cdots \to \yoneda{X^{-2}} \xrightarrow{\yoneda{d_X^{-2}}} \yoneda{X^{-1}} \xrightarrow{\yoneda{d_X^{-1}}} \yoneda{X^0} \to \sF \to 0
\end{equation*}
is acyclic in $\Mod{\cat{E}}/\Eff{\cat{E}}$. In particular,
\[
\coker(\yoneda{d_X^{-2}}) \to \yoneda{X^0} \to \sF
\]
is a $\Qlcat{\cat{E}}$-conflation. Clearly, $\coker(\yoneda{d_X^{-2}})$ lies in the desired categories.
\end{proof}

\begin{proposition} \label{loc_in_Qlcat_extclsd}
Let $\cat{E}$ be a small exact category. The categories $\Pmod{\cat{E}}/\eff{\cat{E}}$ and $\Pmod[b]{\cat{E}}/\eff{\cat{E}}$ are extension-closed in $\Qlcat{\cat{E}}$. 
\end{proposition}
\begin{proof}
We let $* \in \{\emptyset, b\}$. Let $\sE \to \sF \to \sG$ be a $\Qlcat{\cat{E}}$-conflation with $\sE, \sG \in \Pmod[*]{\cat{E}}$. We want to show that $\sF \in \Pmod[*]{\cat{E}}$. 

\begin{step} \label{loc_in_Qlcat_extclsd:step_sE_in_E}
If $\sE \cong \yoneda{M}$ and $\sG \in \Pmod[*]{\cat{E}}$, then $\sF \in \Pmod[*]{\cat{E}}$.
\end{step}

Let $Z$ be a (bounded) $\ext{\cat{E}}$-resolution of $\sG$. We consider the commutative diagram
\begin{equation*}
\begin{tikzcd}[row sep=large]
& \coker(\yoneda{d_Z^{-2}}) \ar[r,"="] \ar[d,infl=\Qlcat{\cat{E}}] & \coker(\yoneda{d_Z^{-2}}) \ar[d,infl=\Qlcat{\cat{E}}] \\
\yoneda{M} \ar[r,infl=\Qlcat{\cat{E}}] \ar[d,"="] & \sP \ar[r,defl=\Qlcat{\cat{E}}] \ar[d,defl=\Qlcat{\cat{E}}] & \yoneda{Z^{0}} \ar[d,defl=\Qlcat{\cat{E}}] \\
\yoneda{M} \ar[r,infl=\Qlcat{\cat{E}}] & \sF \ar[r,defl=\Qlcat{\cat{E}}] & \sG
\end{tikzcd}
\end{equation*}
where the bottom right square is a pullback square and each three-term row or column is a $\Qlcat{\cat{E}}$-conflation. By \cite[Proposition~5]{Rump:2020}, the Yoneda functor induces a fully exact embedding $\cat{E} \to \Qlcat{\cat{E}}$. Hence $\sP \cong \yoneda{Y^{0}}$ and by setting $Y^{n} \colonequals Z^n$ for $n < 0$ we obtain an $\ext{\cat{E}}$-resolution of $\sF$. We have shown $\sF \in \Pmod[*]{\cat{E}}$. 

\begin{step} \label{loc_in_Qlcat_extclsd:step_defl}
For any $\Qlcat{\cat{E}}$-deflation $\yoneda{L} \to \sE$, there exists a $\Qlcat{\cat{E}}$-deflation $\yoneda{N} \to \sG$ with kernel in $\Pmod[*]{\cat{E}}/\eff{\cat{E}}$ and a commutative diagram
\begin{equation*}
\begin{tikzcd}[row sep=large]
\yoneda{L} \ar[r,infl=\cat{E}] \ar[d,defl=\Qlcat{\cat{E}}] & \yoneda{L \oplus N} \ar[r,defl=\cat{E}] \ar[d,defl=\Qlcat{\cat{E}}] & \yoneda{N} \ar[d,defl=\Qlcat{\cat{E}}] \\
\sE \ar[r,infl=\Qlcat{\cat{E}}] & \sF \ar[r,defl=\Qlcat{\cat{E}}] & \sG
\end{tikzcd}
\end{equation*}
in which each row is a $\Qlcat{\cat{E}}$-conflation and each vertical morphism a $\Qlcat{\cat{E}}$-deflation.
\end{step}

By \cref{generating_Rmod}, there exists a $\Qlcat{\cat{E}}$-deflation $\tilde{p} \colon \yoneda{\tilde{N}} \to \sG$ with kernel in $\Pmod[*]{\cat{E}}/\eff{\cat{E}}$. We take the pullback of $\tilde{p}$ along $\sF \to \sG$ to obtain the commuting diagram
\begin{equation*}
\begin{tikzcd}[column sep=large,row sep=large]
\sE \ar[r,infl=\Qlcat{\cat{E}}] \ar[d,"="] & \sP \ar[r,defl=\Qlcat{\cat{E}}] \ar[d,defl=\Qlcat{\cat{E}}] & \yoneda{\tilde{N}} \ar[d,defl=\Qlcat{\cat{E}}] \\
\sE \ar[r,infl=\Qlcat{\cat{E}}] & \sF \ar[r,defl=\Qlcat{\cat{E}}] & \sG \nospacepunct{.}
\end{tikzcd}
\end{equation*}
As $\sP \in \Qlcat{\cat{E}}$, there exists a $\Qlcat{\cat{E}}$-deflation $\yoneda{N} \to \sP$. 
By \cite[Proposition~5]{Rump:2020}, the Yoneda functor induces an fully exact embedding $\cat{E} \to \Qlcat{\cat{E}}$ and hence $\yoneda{N} \to \yoneda{\tilde{N}}$ is the image of an $\cat{E}$-deflation $N \to \tilde{N}$. By \cref{loc_in_Qlcat_extclsd:step_sE_in_E} and the snake lemma, the kernel of the composition $\yoneda{N} \to \sG$ lies in $\Pmod[*]{\cat{E}}/\eff{\cat{E}}$. 
We now obtain the commutative diagram
\begin{equation*}
\begin{tikzcd}[column sep=huge,row sep=large]
\yoneda{L} \ar[r,infl=\cat{E}] \ar[d,defl=\Qlcat{\cat{E}}] & \yoneda{L \oplus N} \ar[r,defl=\cat{E}] \ar[d,defl=\Qlcat{\cat{E}}] & \yoneda{N} \ar[d,"="] \\
\sE \ar[d,"="] \ar[r,infl=\Qlcat{\cat{E}}] & \sE \oplus \yoneda{N} \ar[r,defl=\Qlcat{\cat{E}}] \ar[d,defl=\Qlcat{\cat{E}}] & \yoneda{N} \ar[d,defl=\Qlcat{\cat{E}}] \\
\sE \ar[r,infl=\Qlcat{\cat{E}}] & \sF \ar[r,defl=\Qlcat{\cat{E}}] & \sG
\end{tikzcd}
\end{equation*}
where each row is a $\Qlcat{\cat{E}}$-conflation and each vertical morphism is a $\Qlcat{\cat{E}}$-deflation. 

\begin{step}
If $\sE, \sG \in \Pmod{\cat{E}}$, then $\sF \in \Pmod{\cat{E}}$.
\end{step}
Let $X$ be an $\ext{\cat{E}}$-resolution of $\sE$. We set $\sE^n \colonequals \coker(\yoneda{d_X^{n-1}})$ for $n \leq 0$ and $\sF^0 \colonequals \sF$ and $\sG^0 \colonequals \sG$. For $n \leq 0$ we successively apply \cref{loc_in_Qlcat_extclsd:step_defl} to the $\Qlcat{\cat{E}}$-deflation $p^n \colon \yoneda{X^n} \to \sE^n$ and the $\Qlcat{\cat{E}}$-conflation $\sE^n \to \sF^n \to \sG^n$. The kernels of the $\Qlcat{\cat{E}}$-deflations yield a $\Qlcat{\cat{E}}$-conflation $\sE^{n-1} \to \sF^{n-1} \to \sG^{n-1}$ with $\sG^{n-1} \in \Pmod{\cat{E}}$, as the kernel of $p^n$ is $\sE^{n-1} \to \yoneda{X^n}$ by \cref{ext_acyc_in_loc_acyc}. So we obtain a $\Qlcat{\cat{E}}$-acyclic sequence
\begin{equation*}
\cdots \to \yoneda{Y^{-1}} \to \yoneda{Y^{0}} \to \sF\,.
\end{equation*}
By \cref{ext_acyc_in_loc_acyc}, the complex $Y$ is an $\ext{\cat{E}}$-resolution of $\sF$ and hence $\sF \in \Pmod{\cat{E}}$. 

\begin{step}
If $\sE, \sG \in \Pmod[b]{\cat{E}}$, then $\sF \in \Pmod[b]{\cat{E}}$.
\end{step}
If $\sE \in \Pmod[b]{\cat{E}}$, then we may assume the $\ext{\cat{E}}$-resolution $X$ in the previous step is bounded. This means there exists a $n \leq 0$ with $\sE^n = 0$ and hence $\sF^n \cong \sG^n$. So any $\ext{\cat{E}}$-resolution of $\sG^n$ is also an $\ext{\cat{E}}$-resolution of $\sF^n$. As $\sG^n \in \Pmod[b]{\cat{E}}$, there exists a bounded $\ext{\cat{E}}$-resolution. In particular, this yields a bounded $\ext{\cat{E}}$-resolution of $\sF$ and $\sF \in \Pmod[b]{\cat{E}}$. 
\end{proof}

\begin{definition}
Let $\cat{E}$ be a small exact category. We denote by
\begin{enumerate}
\item $\Rclosure{\cat{E}}$ the category $\Pmod{\cat{E}}/\eff{\cat{E}}$ equipped with the fully exact structure in $\Qlcat{\cat{E}}$; and
\item $\Rclosure[b]{\cat{E}}$ the category $\Pmod[b]{\cat{E}}/\eff{\cat{E}}$ equipped with the fully exact structure in $\Qlcat{\cat{E}}$.
\end{enumerate}
\end{definition}

\begin{lemma} \label{Rcat_resolving}
Let $\cat{E}$ be a small exact category. Then
\begin{enumerate}
\item \label{Rcat_resolving:all} $\cat{E} \subseteq \Rclosure{\cat{E}}$ is a resolving subcategory; 
\item \label{Rcat_resolving:infty} $\cat{E} \subseteq \Rclosure[b]{\cat{E}}$ is a finitely resolving subcategory. 
\end{enumerate}
\end{lemma}
\begin{proof}
As $\cat{E} \subseteq \Qlcat{\cat{E}}$ is semi-resolving by \cite[Theorem~4]{Rump:2020} and $\Rclosure{\cat{E}} \subseteq \Qlcat{\cat{E}}$ is a fully exact subcategory, $\cat{E}$ is resolving in $\Rclosure{\cat{E}}$ and $\Rclosure[b]{\cat{E}}$. By \cref{ext_acyc_in_loc_acyc}, $\cat{E}$ is finitely resolving in $\Rclosure[b]{\cat{E}}$. 
\end{proof}

\subsection{For resolving subcategories}

Let $\cat{E} \subseteq \cat{F}$ be a fully exact subcategory. Restriction yields a functor
\begin{equation*}
r \colon \Mod{\cat{F}} \to \Mod{\cat{E}}\,.
\end{equation*}
As pullbacks and pushouts in $\Mod{\cat{E}}$ and $\Mod{\cat{F}}$ are the pointwise pullbacks and pushouts, respectively, the functor $r$ preserves pullbacks and pushouts.

\begin{lemma} \label{rest_preserves_Eff}
Let $\cat{F}$ be a small exact category and $\cat{E} \subseteq \cat{F}$ be a resolving subcategory. 
For $\sF \in \Mod{\cat{F}}$ one has $\sF \in \Eff{\cat{F}}$ if and only if $r(\sF) \in \Eff{\cat{E}}$. 
\end{lemma}
\begin{proof}
For the forward direction assume $\sF\in \Eff{\cat{F}}$. Let $N\in \cat{E}$ and $x \in r(\sF)(N)$. As $x \in r(\sF)(N) = \sF(N)$, there exists by assumption an $\cat{F}$-deflation $p\colon M \to N$ with $M\in \cat{F}$ and $\sF(p)(x) = 0$. As $\cat{E} \subseteq \cat{F}$ is resolving, there exists an $\cat{F}$-deflation $q \colon L \to M$ with $L\in \cat{E}$. Then the composition $pq$ is an $\cat{E}$-deflation and $r(\sF)(pq)(x) = \sF(pq)(x) = 0$. Hence $r(\sF) \in \Eff{\cat{E}}$. 

For the backward direction assume $r(\sF) \in \Eff{\cat{E}}$. Let $N \in \cat{F}$ and $x \in \sF(N)$. Since $\cat{E}$ is a resolving subcategory there exists a $\cat{F}$-deflation $p \colon M \to N$ with $M \in \cat{E}$. We consider $\sF(p)(x) \in \sF(M) = r(\sF)(M)$. By assumption, there exists an $\cat{E}$-deflation $q \colon L \to M$ such that $0 = r(\sF)(q)(\sF(p)(x)) = \sF(pq)(x)$. 
\end{proof}

As a consequence of this \namecref{rest_preserves_Eff}, there is an induced functor
\begin{equation*}
\bar{r} \colon \Mod{\cat{F}}/\Eff{\cat{F}} \to \Mod{\cat{E}}/\Eff{\cat{E}}\,,
\end{equation*}
that preserves pullbacks and pushouts. 

\begin{proposition} \label{resolving_same_Qlcat}
Let $\cat{F}$ be a small exact category and $\cat{E} \subseteq \cat{F}$ be a resolving subcategory. Then $\bar{r}$ restricts to an equivalence $\Qlcat{\cat{F}} \to \Qlcat{\cat{E}}$. 
\end{proposition}
\begin{proof}
To avoid confusion, we write
\begin{equation*}
\yoneda[\cat{E}]{-} \colon \cat{E} \to \Mod{\cat{E}} \quad \text{and} \quad \yoneda[\cat{F}]{-} \colon \cat{F} \to \Mod{\cat{F}}
\end{equation*}
for the Yoneda functors on $\cat{E}$ and $\cat{F}$, respectively. We observe that $r(\yoneda[\cat{F}]{M}) = \yoneda[\cat{E}]{M}$ when $M \in \cat{E}$.

\begin{step}\label{step:modE_maps_lift_to_modF}
Let $\alpha\colon \sF\to \sG$ be a morphism in $\mod{\cat{E}}$. Then there exists a morphism $\widehat{\alpha}\colon \widehat{\sF}\to \widehat{\sG}$ in $\mod{\cat{F}}$ such that $r(\widehat{\alpha})=\alpha$. Additionally, composition is preserved under this lift from $\mod{\cat{E}}$ to $\mod{\cat{F}}$.
\end{step}

Let $\sF = \coker(\yoneda[\cat{E}]{f})$ and $\sG = \coker(\yoneda[\cat{E}]{g})$. We find the commutative diagram

\[
\begin{tikzcd}
\yoneda[\cat{E}]{L} \arrow[d] \arrow[r, "{\yoneda[\cat{E}]{f}}"] & \yoneda[\cat{E}]{L'} \arrow[r, "c"] \arrow[d] & \coker(\yoneda[\cat{E}]{f}) \arrow[d, "\alpha"] \\
\yoneda[\cat{E}]{M}\arrow[r, "{\yoneda[\cat{E}]{g}}"] & \yoneda[\cat{E}]{M'}\arrow[r]  & \coker(\yoneda[\cat{E}]{g})
\end{tikzcd}\]
in $\mod{\cat{E}}$. Using Yoneda's lemma, we may lift the left square to a square in $\cat{E}$. We then embed it into $\cat{F}$ under the inclusion functor and apply $\yoneda[\cat{F}]{-}$. Composition is preserved, and we find an induced map $\widehat{\alpha}\colon \coker(\yoneda[\cat{F}]{f})\to \coker(\yoneda[\cat{F}]{g})$ which satisfies the claim. 

\begin{step} \label{step:lift_coker_to_E_mor}
For any morphism $f$ in $\cat{F}$, there exists a morphism $g$ in $\cat{E}$ such that $\coker(\yoneda[\cat{F}]{f}) \cong \coker(\yoneda[\cat{F}]{g})$ in $\Qlcat{\cat{F}}$.
\end{step}

Let $f\colon L \to M$ be a morphism in $\cat{F}$. Since $\cat{E}$ is resolving in $\cat{F}$, we find an $\cat{F}$-deflation $p\colon \widehat{M} \to M$ with $\widehat{M}\in \cat{E}$. Form the pullback of $f$ along $p$ to obtain a morphism $\widehat{f}\colon P\to \widehat{M}$ and an $\cat{F}$-deflation $P\to L$. Take another $\cat{F}$-deflation $q \colon \widehat{L} \to P$ with $\widehat{L} \in \cat{E}$. We have a morphism $g \colonequals \widehat{f} q \colon \widehat{L} \to \widehat{M}$ in $\cat{E}$. As the Yoneda functor preserves pullbacks, we observe that $\coker(\yoneda[\cat{F}]{f}) = \coker(\yoneda[\cat{F}]{\widehat{f}})$. Since $\yoneda[\cat{F}]{q}$ is a $\Qlcat{\cat{E}}$-deflation, it follows that $\coker(\yoneda[\cat{F}]{\widehat{f}}) \cong \coker(\yoneda[\cat{F}]{g})$. 

\begin{step}\label{step:equal_E-cokernels_equal_F-cokernels}
Let $f$ and $g$ be morphisms in $\cat{E}$. If $\coker(\yoneda[\cat{E}]{f}) = \coker(\yoneda[\cat{E}]{g})$ in $\mod{\cat{E}}$ then $\coker(\yoneda[\cat{F}]{f}) = \coker(\yoneda[\cat{F}]{g})$ in $\mod{\cat{F}}$.
\end{step}

By \cref{step:modE_maps_lift_to_modF}, we can lift any morphism $\coker(\yoneda[\cat{E}]{f}) \to \coker(\yoneda[\cat{E}]{g})$ to a morphism of their presentations. The morphism on the presentation is unique up to homotopy. The claim follows as we can lift the homotopies to $\cat{E}$ using Yoneda's lemma.

\begin{step} \label{same_Qlcat:essentially_surj}
The functor $\bar{r}$ restricts to a functor $\Qlcat{\cat{F}} \to \Qlcat{\cat{E}}$; that is, if $\sF \in \Qlcat{\cat{F}}$, then $\bar{r}(\sF) \in \Qlcat{\cat{E}}$.
\end{step}

Using first \cref{step:lift_coker_to_E_mor}, let $\sF = \coker(\yoneda[\cat{F}]{g})$ with $g$ a morphism in $\cat{E}$. As $r(\yoneda[\cat{F}]{g}) = \yoneda[\cat{E}]{g}$ and $\bar{r}$ preserves cokernels, we have 
\begin{equation} \label{eq:rest_qlcat}
    \bar{r}(\coker(\yoneda[\cat{F}]{g})) = \coker(\bar{r}(\yoneda[\cat{F}]{g})) = \coker(\yoneda[\cat{E}]{g})\in \Qlcat{\cat{E}}.
\end{equation}

\begin{step}
The restriction of $\bar{r}$ to $\Qlcat{\cat{F}} \to \Qlcat{\cat{E}}$ is essentially surjective; that is, if $\sE \in \Qlcat{\cat{E}}$, then there exists $\sF \in \Qlcat{\cat{F}}$ with $\sE \cong \bar{r}(\sF)$.
\end{step}

Let $\sE\in \Qlcat{\cat{E}}$ be of the form $\coker(\yoneda[\cat{E}]{g})$ for some morphism $g\in \cat{E}$. Then by \cref{eq:rest_qlcat}, $\sF = \coker(\yoneda[\cat{F}]{g})$ satisfies the claim.

\begin{step}\label{step:full}
The restriction of $\bar{r}$ to $\Qlcat{\cat{F}} \to \Qlcat{\cat{E}}$ is full.
\end{step}

By \cref{step:modE_maps_lift_to_modF} any morphism in $\mod{\cat{E}}$ lifts to a morphism in $\mod{\cat{F}}$. For a morphism $\alpha$ in $\mod{\cat{F}}$ we have the following equivalences
\begin{equation*}
\begin{aligned}
&\quad\quad \alpha \in S_{\eff{\cat{F}}} \text{ (that is invertible in $\Qlcat{\cat{F}}$)} \\
&\iff \ker(\alpha) , \coker(\alpha) \in \Eff{\cat{F}} \\
&\iff \ker(r(\alpha)), \coker(r(\alpha)) \in \Eff{\cat{E}} \\
&\iff r(\alpha) \in S_{\eff{\cat{E}}} \text{ (that is invertible in $\Qlcat{\cat{E}}$)}\,.
\end{aligned}
\end{equation*}
For the first and third equivalence we use the localization of the ambient abelian category and the second equivalence holds by \cref{rest_preserves_Eff}. 

\begin{step}
The restriction of $\bar{r}$ to $\Qlcat{\cat{F}} \to \Qlcat{\cat{E}}$ is faithful.
\end{step}

To see that $\bar{r}$ is faithful, let $\eta\colon \sE\to \sF$ be a morphism in $\Qlcat{\cat{F}}$ such that $\bar{r}(\eta) = 0$ in $\Qlcat{\cat{E}}$. Suppose that $\eta$ is represented by the zigzag $\sE\xleftarrow{\alpha}\sG\xrightarrow{\beta}\sF$ with $\alpha \in S_{\eff{\cat{E}}}$. 
There is a morpshim $\sH\xrightarrow{\gamma}\sG$ in $S_{\eff{\cat{E}}}$, and there is a homotopy witnessing $\beta\gamma=0$ in $\mod{\cat{E}}$. By a similar argument as in \cref{step:modE_maps_lift_to_modF}, the homotopy lifts to $\mod{\cat{F}}$. By \cref{rest_preserves_Eff,eff_is_Eff_cap_mod} the lift of $\gamma$ to $\mod{\cat{F}}$ lies in $S_{\eff{\cat{F}}}$.
\end{proof}

\begin{lemma} \label{resolving_ext_resn}
Let $\cat{F}$ be a small exact category and $\cat{E} \subseteq \cat{F}$ a resolving subcategory. For an $\ext{\cat{F}}$-resolution $X$ there exists an $\ext{\cat{E}}$-resolution $W$ and an $\cat{F}$-quasi-isomorphism $f \colon W \to X$.
\end{lemma}
\begin{proof}
Using induction on the cohomological degree, we construct an $\ext{\cat{E}}$-resolu\-tion $W$ and an $\cat{F}$-quasi-isomorphism $f \colon W \to X$. 

Set $W^n = 0$ and $f^n = 0$ for $n > 0$. We  claim, that for any $n$ there exists a commutative diagram
\begin{equation*}
\begin{tikzcd}[column sep=large]
& P^{n} \ar[r,defl=\cat{F},"q^{n}" {near start}] \ar[dl,"b^{n}"] & X^{n} \ar[d,"d_X^{n}"] \ar[dl,"a^{n}"] \\
W^{n+1} \ar[r,defl=\cat{F},"p^{n+1}"] & P^{n+1} \ar[r,defl=\cat{F},"q^{n+1}"] & X^{n+1}
\end{tikzcd}
\end{equation*}
where $P^{n}$ is the pullback of the span $W^{n+1} \to P^{n+1} \leftarrow X^{n}$ and $d_X^{n-1} a^{n} = 0$ and $W^{n+1} \in \cat{E}$.

For $n \geq 0$, we set $P^n \colonequals X^n$, $a^n = b^n = d_X^n$ and $q^n = \id_{X^n}$ and $p^{n+1} = \id_{X^{n+1}}$. Clearly, the claim holds. 

We assume the desired morphisms and objects exist for some $n \leq 0$. As $\cat{E} \subseteq \cat{F}$ is resolving, there exists an $\cat{F}$-deflation $p^n \colon W^n \to P^n$ with $W^n \in \cat{E}$ by \cref{resolving:generator}. Since $P^{n}$ is a pullback and $a^n d_X^{n-1}$, there exists a morphism $a^{n-1}$ such that
\begin{equation*}
q^{n} a^{n-1} = d_X^{n-1} \quad \text{and} \quad b^{n} a^{n-1} = 0\,.
\end{equation*}
Further, as 
\begin{equation*}
q^{n} a^{n-1} d_X^{n-2} = d_X^{n-1} d_X^{n-2} = 0 \quad \text{and} \quad b^{n} a^{n-1} d_X^{n-2} = 0
\end{equation*}
we obtain $a^{n-1} d_X^{n-2} = 0$. Let $P^{n-1}$ be the pullback of $a^{n-1}$ and $p^{n}$ with canonical morphisms $b^{n-1}$ and $q^{n-1}$. 

From the above construction we obtain a complex $W$ with $d_W^{n} \colonequals p^{n} b^{n}$ and a morphism $f \colon W \to X$ given by $f^{n} = q^{n} p^{n}$. It is an $\cat{F}$-quasi-isomorphism as the pullback squares yield $\cat{F}$-conflations. By \cref{Extacyc_cone_then} the complex $X$ is left $\ext{\cat{F}}$-acyclic and by \cref{resolving_ext_acyc} it is left $\ext{\cat{E}}$-acyclic.
\end{proof}

\begin{corollary} \label{finitely_resolving_ext_resn}
Let $\cat{F}$ be a small exact category and $\cat{E} \subseteq \cat{F}$ a finitely resolving subcategory. For a bounded $\ext{\cat{F}}$-resolution $X$ there exists a bounded $\ext{\cat{E}}$-resolution $W$ and an $\cat{F}$-quasi-isomorphism $f \colon W \to X$.
\end{corollary}
\begin{proof}
Let $n \leq 0$ such that $X^{<n} = 0$. We construct $W^{> n}$ and $P^{\geqslant n}$ as in the proof of \cref{resolving_ext_resn}. Let $Y$ be a bounded $\ext{\cat{E}}$-resolution that resolves $P^{n}$. Set $W^{n+i} = Y^{i}$ and $d_W^{n+i} = Y^{i}$ for $i \leq 0$. Then $W$ is a bounded $\ext{\cat{E}}$-resolution and $W \to X$ an $\cat{F}$-quasi-isomorphism. 
\end{proof}

\begin{proposition} \label{resolving_same_rclosure}
Let $\cat{F}$ be a small exact category.
\begin{enumerate}
\item If $\cat{E} \subseteq \cat{F}$ is resolving, then $\bar{r}$ restricts to an equivalence $\Rclosure{\cat{F}} \to \Rclosure{\cat{E}}$. 
\item If $\cat{E} \subseteq \cat{F}$ is finitely resolving, then $\bar{r}$ restricts to an equivalence $\Rclosure[b]{\cat{F}} \to \Rclosure[b]{\cat{E}}$. 
\end{enumerate}
\end{proposition}
\begin{proof}
Let $\sF \in \Qlcat{\cat{F}}$. By \cref{resolving_same_Qlcat}, it is enough to show $\sF \in \Rclosure[*]{\cat{F}}$ if and only if $\bar{r}(\sF) \in \Rclosure[*]{\cat{E}}$ for $* \in \{\emptyset, b\}$. 

We assume $\sF \in \Rclosure{\cat{F}}$. Then there exists an $\ext{\cat{F}}$-resolution $X$ with $\sF = \coker(\yoneda{d_X^{-1}})$. By \cref{resolving_ext_resn}, there exists an $\ext{\cat{E}}$-resolution $W$ and an $\cat{F}$-quasi-isomorphism $f \colon W \to X$. Hence
\begin{equation*}
\bar{r}(\sF) = \bar{r}(\coker(\yoneda[\cat{F}]{d_W^{-1}})) = \coker(\yoneda[\cat{E}]{d_W^{-1}})) \in \Rclosure{\cat{E}}\,.
\end{equation*}
When $\cat{E} \subseteq \cat{F}$ is finitely resolving and $\sF \in \Rclosure[b]{\cat{F}}$, then there exists a bounded $\ext{\cat{F}}$-resolution $X$, and the $\ext{\cat{E}}$-resolution $W$ is also bounded by \cref{finitely_resolving_ext_resn}. 

For the converse we assume $\bar{r}(\sF) \in \Rclosure[*]{\cat{E}}$. Let $X$ be a (bounded) $\ext{\cat{E}}$-resolution of $\bar{r}(\sF)$. Then $X$ is an $\ext{\cat{F}}$-resolution by \cref{resolving_ext_acyc}, and $\sF \cong \coker(\yoneda[\cat{F}]{d_X^{-1}})$ by \cref{same_Qlcat:essentially_surj} of \cref{resolving_same_Qlcat}. Hence $X$ is a (bounded) $\ext{\cat{F}}$-resolution of $\sF$ and $\sF \in \Rclosure[*]{\cat{F}}$. 
\end{proof}

Together with \cref{Rcat_resolving} we immediately obtain the following \namecref{Rcat_idempotent}.

\begin{corollary} \label{Rcat_idempotent}
Let $\cat{E}$ be a small exact category. Then there are equivalences
\begin{equation*}
\Rclosure{\Rclosure{\cat{E}}} \to \Rclosure{\cat{E}} \quad \text{and} \quad \Rclosure[b]{\Rclosure[b]{\cat{E}}} \to \Rclosure[b]{\cat{E}}
\end{equation*}
 of exact categories. \qed
\end{corollary}

\section{The resolving completion}

Let $\cat{E}$ be a (small) exact category. In \cref{sec:dcat,sec:rclosure} we have constructed ambient categories $\lheart[*]{\cat{E}}$ and $\Rclosure[*]{\cat{E}}$ in which $\cat{E}$ is (finitely) resolving. They both behave like a completion operation; see \cref{heart_idempotent,Rcat_idempotent}. We use these to obtain a universal property.

\begin{definition}\label{universal_property}
Let $\cat{E} \to \cat{R}$ be a fully exact embedding such that $\cat{E}$ is (finitely) resolving in $\cat{R}$. We call $\cat{R}$ the \emph{(finitely) resolving completion} of $\cat{E}$, if for any fully exact embedding $\cat{E} \to \cat{F}$ such that $\cat{E}$ is (finitely) resolving in $\cat{F}$, there exists a fully exact embedding $\cat{F} \to \cat{R}$, unique up to natural isomorphism, extending $\cat{E} \to \cat{R}$. In particular,  there is a commutative diagram
\[
\begin{tikzcd}
\cat{E} \ar[r,tail] \arrow[rd,tail] & \cat{F} \ar[d,tail,dashed,"\exists i"] \\
& \cat{R} \nospacepunct{.}
\end{tikzcd}
\]
\end{definition}

It is straightforward to check that then the fully exact embedding $\cat{F} \to \cat{R}$ is (finitely) resolving.

\begin{proposition} \label{LH_completion}
Let $\cat{E}$ be an exact category. Then $\lheart{\cat{E}}$ is a resolving completion of $\cat{E}$ and $\lheart[b]{\cat{E}}$ is a finitely resolving completion of $\cat{E}$.
\end{proposition}
\begin{proof}
By \cref{heart_resolving}, $\cat{E}$ is resolving in $\lheart{\cat{E}}$ and finitely resolving in $\lheart[b]{\cat{E}}$. Let $\cat{E} \to \cat{F}$ be a fully exact embedding such that $\cat{E}$ is (finitely) resolving in $\cat{F}$. By \cref{rcat_same_heart}, there are equivalences
\begin{equation*}
\lheart{\cat{E}} \to \lheart{\cat{F}} \quad \text{and} \quad \lheart[b]{\cat{E}} \to \lheart[b]{\cat{F}}
\end{equation*}
of exact categories, respectively. Hence the desired functor is given by $\cat{F} \to \lheart[*]{\cat{F}}$ followed by the quasi-inverse of the above equivalence. 

The uniqueness of this functor follows again from \cref{rcat_same_heart}.
\end{proof}

Replacing \cref{heart_resolving,rcat_same_heart} by \cref{Rcat_resolving,resolving_same_rclosure}, the same proof yields:

\begin{proposition} \label{Rcat_completion}
Let $\cat{E}$ be a small exact category. Then $\Rclosure{\cat{E}}$ is a resolving completion of $\cat{E}$ and $\Rclosure[b]{\cat{E}}$ is a finitely resolving completion of $\cat{E}$. \qed
\end{proposition}

\begin{theorem}\label{Rcat_equals_LH}
Let $\cat{E}$ be a small exact category. Then there is an equivalence of exact categories $\lheart{\cat{E}} \to \Rclosure{\cat{E}}$ that restricts an equivalence $\lheart[b]{\cat{E}} \to \Rclosure[b]{\cat{E}}$ and to the identity on $\cat{E}$. The equivalence is given by $X \mapsto \coker(\yoneda{d_X^{-1}})$. 
\end{theorem}
\begin{proof}
Since $\Rclosure{\cat{E}}$ and $\lheart{\cat{E}}$ are resolving completions of $\cat{E}$, there exists a commutative diagram of fully exact embeddings
\begin{equation*}
\begin{tikzcd}
& \cat{E} \ar[dl,tail] \ar[d,tail] \ar[dr,tail] \\
\lheart{\cat{E}} \ar[r,tail] & \Rclosure{\cat{E}} \ar[r,tail] & \lheart{\cat{E}}
\end{tikzcd}
\end{equation*}
where the composition in the bottom row is an equivalence. The same holds when switching $\lheart{\cat{E}}$ and $\Rclosure{\cat{E}}$. 
\end{proof}

We have already seen some examples of the finitely resolving completion in \cref{sec:dcat}. In some settings the viewpoint of the functor category is more convenient.

\begin{example} \label{rclosure_split_exact}
Let $\cat{A}$ be an additive category equipped with the split exact structure. Then the category of effaceable functors is zero and the Yoneda embedding 
\[
\mathbb{Y} \colon \cat{A} \to \Pmod{\cat{A}}
\]
is the resolving completion. The Yoneda embedding 
\[
\mathbb{Y} \colon \cat{A} \to \Pmod[b]{\cat{A}}
\]
is the finitely resolving completion.
\end{example}

\begin{example}
Let $\cat{E}$ be an exact category with enough projectives and let $\cat{P}$ be its subcategory of projective objects. Then $\cat{P} \subseteq \cat{E}$ is a resolving subcategory when equipped with the split exact structure. Then the restricted Yoneda functor 
\[
\bar{\mathbb{Y}} \colon \cat{E} \to \Pmod{\cat{P}} \,, \quad M \mapsto \left.\Hom_{\cat{E}}(-,M)\right|_{\cat{P}} 
\]
defines a resolving completion of $\cat{E}$. Similarly, if $\cat{E}$ has enough projectives and every object has a finite projective resolution then 
\[
\bar{\mathbb{Y}} \colon \cat{E} \to \Pmod[b]{\cat{P}} \,, \quad M \mapsto \left.\Hom_{\cat{E}}(-,M)\right|_{\cat{P}} 
\]
defines the finite resolving completion.
This follows, from the previous example and Proposition  \ref{resolving_same_rclosure}.
\end{example}

\begin{example} \label{ic-wic}
Let $\cat{E}\to \cat{E}^{ic}$ be the idempotent completion of an exact category; cf.\ \cite[Section~6]{Buehler:2010}. This is a fully exact embedding of a resolving subcategory and therefore, by \cref{Rcat_idempotent} the resolving completion of $\cat{E}$ and $\cat{E}^{ic}$ can be identified and is idempotent complete itself.

Let $\cat{E} \to \cat{E}^{wic}$ be the weak idempotent completion of an exact category; cf.\ \cite[Remark~7.8]{Buehler:2010}. This is a fully exact embedding of a finitely resolving subcategory. Therefore, the finitely resolving completion of $\cat{E}$ and $\cat{E}^{wic}$ can be identified and is weakly idempotent complete itself. 
\end{example}

\begin{example}
An example for 
$\lheart[b]{\cat{E}}\neq \rheart[b]{\lheart[b]{\cat{E}}}$:
Let $\Lambda$ be the path algebra of the quiver 
\[
\begin{tikzcd}
1 \arrow[r] & 2 \arrow[r, shift left=5] \arrow[r, shift left=4]\arrow[r, "\vdots"', shift left=3] 
& 3
\end{tikzcd}
\]
with countably infinitely many arrows from $2$ to $3$ module all paths of length $2$ over a field. Let $P(1),P(2)$ and $P(3)$ be the indecomposable projective quiver representations with simple top $S(1)$, $S(2)$ and $S(3)$, respectively. We consider $\cat{P} =\add(P(1) \oplus P(2) \oplus P(3))$ be subcategory of finitely generated projective quiver representations of $\mod{\Lambda}$.

For $\lheart[b]{\cat{P}} \neq \rheart[b]{\lheart[b]{\cat{P}}}$ it is enough to find an epimorphism in $\lheart[b]{\cat{P}}$ that is not an $\lheart[b]{\cat{P}}$-deflation by the dual of \cref{char_lheart_mono}. We use the identification $\lheart[b]{\cat{P}} = \Pmod[b]{\cat{P}}$. Let $a \colon P(2) \to P(1)$ be the composition of the canonical morphisms $P(2) \to S(2) \to P(1)$. We show that $a$ satisfies the desired conditions. In the category $\mod{\Lambda}$, we have $\coker(a) = S(1)$ and $\ker(a) \cong S(3)^{\oplus \BN}$. This implies that $S(1)$ is not in $\Pmod[b]{\cat{P}}$. Next, we observe that $S(1)$ is an injective object, so every object $M$ with a non-zero morphism $S(1) \to M$ must have $S(1)$ as a summand. But objects in $\Pmod[b]{\cat{P}}$ have no summand $S(1)$, therefore $a$ is an epimorphism in $\Pmod[b]{\cat{P}}$. It is not a $\Pmod[b]{\cat{P}}$-deflation because it has no kernel in $\Pmod[b]{\cat{P}}$. 
\end{example}

\subsection{The n-resolving completion} \label{n_res_completion}

The constructions of the finite resolving completion $\lheart[b]{\cat{E}}$ and $\Rclosure[b]{\cat{E}}$ can be further restricted to those resolutions that have a universal bound.

We define
\begin{equation*}
\cat{V}_\ell^n(\cat{E}) \coloneqq \cat{V}_\ell(\cat{E}) \cap \susp^n \cat{V}(\cat{E}) \quad \text{and} \quad \cat{U}_r^n(\cat{E}) \coloneqq \cat{U}_r(\cat{E}) \cap \susp^{-n} \cat{U}(\cat{E})\,.
\end{equation*}
We denote by $\lheart[n]{\cat{E}}$ the category $\cat{U}(\cat{E}) \cap \cat{V}_\ell^n(\cat{E})$ equipped with the admissible exact structure; and by $\rheart[n]{\cat{E}}$ the category $\cat{U}_r^n(\cat{E}) \cap \cat{V}(\cat{E})$ equipped with the admissible exact structure.

For any $n \geq 0$, the category $\lheart[n]{\cat{E}}$ contains $\cat{E}$ as a finitely resolving subcategory. Moreover, $\cat{E}$ is a \emph{$n$-resolving} subcategory: A fully exact subcategory $\cat{E}$ of an exact category $\cat{F}$ is called $n$-resolving in $\cat{F}$ if $\cat{E}$ is resolving in $\cat{F}$, and for any $M$ in $\cat{F}$, the induced $\cat{F}$-acyclic sequence in \cref{eq:resolving_resn} is of length $n$; that is $X^i = 0$ for $i < n$. By \cite[Theorem~4.1]{Henrard/vanRoosmalen:2020c}, there is a derived equivalence \begin{equation*}
\dcat{\cat{E}} \cong \dcat{\lheart[n]{\cat{E}}}
\end{equation*}
for all $n\geq1$.

Similarly, we can also specialize the construction of $\Rclosure[b]{\cat{E}}$: We define $\Pmod[n]{\cat{E}}$ as the subcategory of $\Pmod[]{\cat{E}}$ which consist of functors with $\ext{\cat{E}}$-resolution $X$ with $X^i = 0$ for $i < -n$. Analogously to \cref{eff_defl_percolating,loc_in_Qlcat_extclsd}, $\eff{\cat{E}}$ is a deflation-percolating subcategory of $\Pmod[n]{\cat{E}}$ for $n\geq 2$ and the localization $\Pmod[n]{\cat{E}}/\eff{\cat{E}}$ is extension-closed in $\Qlcat{\cat{E}}$. We therefore define categories $\Rclosure[n]{\cat{E}}$ as the category $\Pmod[n]{\cat{E}}/\eff{E}$ equipped with the fully exact structure in $\Qlcat{\cat{E}}$ for all $n\geq 2$. 

For $n = 1$, we cannot localize $\Pmod[1]{\cat{E}}$ at the $\cat{E}$-effaceable functors, as the $\cat{E}$-effaceable functors are not contained in $\Pmod[1]{\cat{E}}$. To rectify this, we replace $\Pmod[1]{\cat{E}}$ by the full subcategory $\Pmodhat[1]{\cat{E}}$ of $\mod{\cat{E}}$ consisting of functors of the form $\sF=\coker{\yoneda[]{f}}$, where $f$ is left $\cat{E}$-admissible; that is $f$ factors as an $\cat{E}$-deflation followed by a monomorphism. One can show that $\eff{E}$ is deflation-percolating in this larger category, and the localization is extension-closed in $\Qlcat{\cat{E}}$. We denote by $\Rclosure[1]{\cat{E}}$ category $\Pmodhat[1]{\cat{E}}/\eff{\cat{E}}$ equipped with the fully exact structure in $\Qlcat{\cat{E}}$.

For all $n\geq1$, it holds that $\cat{E}$ is an $n$-resolving subcategory of $\Rclosure[n]{\cat{E}}$. Moreover, if we define the $n$-resolving completion analogously to \cref{universal_property}, then the proofs of \cref{LH_completion,Rcat_completion,Rcat_equals_LH} can be adjusted to show:

\begin{theorem}\label{Rcat_equals_LH_restriction}
Let $\cat{E}$ be a (small) exact category. Then $\lheart[n]{\cat{E}}$ and $\Rclosure[n]{\cat{E}}$ are $n$-resolving completions for all $n \geq 1$.

Moreover, there is an equivalence $\lheart[n]{\cat{E}}\cong \Rclosure[n]{\cat{E}}$ of exact categories for all $n\geq 1$.
\end{theorem}

Previously, the 1-resolving and 2-resolving completions have been described in special cases when they provide t-structures; see \cref{t_str_iff_Ext_kernel}.

\begin{example}\label{example:quasi-abelian-heart}
Let $\cat{E}$ be a quasi-abelian category equipped with the maximal exact structure. Then $\mod{\cat{E}} = \Pmodhat[1]{\cat{E}}$. In particular, this yields
\begin{equation*}
\lheart{\cat{E}} \cong \Rclosure{\cat{E}} = \Rclosure[1]{\cat{E}} \cong \lheart[1]{\cat{E}}\,;
\end{equation*}
the latter is the full subcategory category of $\dcat{\cat{E}}$ consisting of two-term complexes concentrated in degrees $0,-1$ given by a monomorphism. This is the left heart as described in \cite[Section~1.2.2]{Schneiders:1999}.
\end{example}

\begin{example}\label{example:kernels-cat-heart}
Let $\cat{E}$ be an exact category with kernels. Then $\mod{\cat{E}} = \Pmod[2]{\cat{E}}$. In particularly, this yields 
\begin{equation*}
\lheart{\cat{E}} \cong \Rclosure{\cat{E}} = \Rclosure[2]{\cat{E}} \cong \lheart[2]{\cat{E}}\,;
\end{equation*}
the latter is the left heart as described in \cite[Corollary~3.10]{Henrard/Kvamme/vanRoosmalen/Wegner:2023}. It is the full subcategory of $\dcat[]{\cat{E}}$ consisting of three-term complexes of the form $\ker{f}\to L\xrightarrow{f}M$ concentrated in degrees $0,-1,-2$.
\end{example}

There are plenty of exact categories that do not have kernels. For example, the category of finitely presented modules over a ring that is not left coherent. For other categories, like the category of complete $\mathsf{LB}$-spaces, it is not known whether kernels exist; see \cite[Section~5]{Wegner:2025}. 

\begin{remark}
We expect that most of our results can be generalized to one-sided exact categories as introduced in \cite{Rump:2011,Bazzoni/Crivei:2013}, similarly to \cite{Henrard/Kvamme/vanRoosmalen/Wegner:2023}.
This will be done in the PhD thesis of the first author.
\end{remark}

\bibliographystyle{amsalpha}
\bibliography{The_Resolving_Completion_of_an_Exact_Category}

\end{document}